\documentclass[11pt,reqno,oneside,a4paper]{article}
\usepackage[a4paper,includeheadfoot,left=25mm,right=25mm,top=00mm,bottom=20mm,headheight=20mm]{geometry}
\usepackage{amssymb,amsmath,amsthm}
\usepackage{xcolor,graphicx}
\usepackage{verbatim}
\usepackage{hyperref}
\usepackage{mathtools}
\usepackage{titling}
\usepackage{fancyhdr}
\newcommand{\runningtitle}{Running Title}
\newtheorem{thm}{Theorem}[section]
\newtheorem{lem}[thm]{Lemma}

\newtheorem{prop}[thm]{Proposition}

\theoremstyle{definition}
\newtheorem{defn}[thm]{Definition}

\newtheorem{ntn}[thm]{Notation}

\newtheorem*{prob*}{Problem}
\theoremstyle{remark}
\newtheorem{rmk}[thm]{Remark}

\numberwithin{equation}{section}

\newcommand{\NN}{\mathbb N}              
\newcommand{\RR}{\mathbb R}              
\newcommand{\CC}{\mathbb C}              
\renewcommand{\Re}{\operatorname*{Re}} 
\newcommand{\D}{\ensuremath{\,\mathrm{d}}}
\newcommand{\ri}{\ensuremath{\mathrm{i}}}
\newcommand{\re}{\ensuremath{\mathrm{e}}}
\newcommand{\la}{\ensuremath{\lambda}}
\renewcommand{\epsilon}{\varepsilon}
\renewcommand{\geq}{\geqslant}
\renewcommand{\leq}{\leqslant}

\providecommand{\clos}{\operatorname{clos}}

\providecommand{\argdot}{{}\cdot{}}

\newcommand{\Mspacer}{\;} 
\newcommand{\M}[3]{#1_{#2\Mspacer#3}} 
\newcommand{\Msup}[4]{#1_{#2\Mspacer#3}^{#4}} 

\newcommand{\Caputo}[2]{\left(\prescript{\ensuremath{\mathrm{C}}}{}{\ensuremath{\mathrm{D}}}_{0\Mspacer+}^{#1}#2\right)}
\newcommand{\CaputoSeq}[3]{\left(\prescript{\ensuremath{\mathrm{C}}}{}{\ensuremath{\mathrm{D}}}_{0\Mspacer+}^{#1\Mspacer#2}#3\right)}

\usepackage{scalerel}
\usepackage{stackengine}
\stackMath
\newcommand\reallywidecheck[1]{%
\savestack{\tmpbox}{\stretchto{%
  \scaleto{%
    \scalerel*[\widthof{\ensuremath{#1}}]{\kern-.6pt\bigwedge\kern-.6pt}%
    {\rule[-\textheight/2]{1ex}{\textheight}}
  }{\textheight}%
}{0.5ex}}%
\stackon[1pt]{#1}{\scalebox{-1}{\tmpbox}}%
}

\newcommand\reallywidehat[1]{%
\savestack{\tmpbox}{\stretchto{%
  \scaleto{%
    \scalerel*[\widthof{\ensuremath{#1}}]{\kern-.6pt\bigwedge\kern-.6pt}%
    {\rule[-\textheight/2]{1ex}{\textheight}}
  }{\textheight}%
}{0.5ex}}%
\stackon[1pt]{#1}{\tmpbox}%
}

\author{David A. Smith \& Wei Yang Toh}
\title{Linear evolution equations on the half line with dynamic boundary conditions}
\renewcommand{\runningtitle}{Dynamic BC. Half line}
\date{\today}

\begin{document}
\maketitle
\thispagestyle{fancy}

\begin{abstract}
    The classical half line Robin problem for the heat equation may be solved via a spatial Fourier transform method.
    In this work, we study the problem in which the static Robin condition $bq(0,t)+q_x(0,t)=0$ is replaced with a dynamic Robin condition; $b=b(t)$ is allowed to vary in time.
    Applications include convective heating by a corrosive liquid.
    We present a solution representation, and justify its validity, via an extension of the Fokas transform method.
    We show how to reduce the problem to a variable coefficient fractional linear ordinary differential equation for the Dirichlet boundary value.
    We implement the fractional Frobenius method to solve this equation, and justify that the error in the approximate solution of the original problem converges appropriately.
    We also demonstrate an argument for existence and unicity of solutions to the original dynamic Robin problem for the heat equation.
    Finally, we extend these results to linear evolution equations of arbitrary spatial order on the half line, with arbitrary linear dynamic boundary conditions.
\end{abstract}

\tableofcontents

\section{Introduction} \label{sec:Introduction}

We study initial dynamic boundary value problems (IdBVP) for linear evolution equations in $(1+1)$ dimensions on the half line, in which not only the boundary data, but the boundary forms themselves, are permitted to vary in time.
Specifically, we study problems of class
\begin{subequations} \label{eqn:IdBVP}
\begin{align}
    \label{eqn:IdBVP.PDE} \tag{\theparentequation.PDE}
    \left[ \partial_t + a (-\ri\partial_x)^n \right] q(x,t) &= 0 & (x,t) &\in (0,\infty)\times(0,T), \\
    \label{eqn:IdBVP.IC} \tag{\theparentequation.IC}
    q(x,0) &= q_0(x) & x &\in [0,\infty), \\
    \label{eqn:IdBVP.dBC} \tag{\theparentequation.dBC}
    \sum_{j=0}^{n-1} \M{b}{k}{j}(t) \partial_x^j q(0,t) &= h_k(t), & t &\in [0,T], \; k\in\{1,2,\ldots,N\},
\end{align}
\end{subequations}
in which $2\leq n\in\NN$, $a\in\CC$ with $\lvert a \rvert=1$ and $n,a,N$ obey the criteria
\begin{equation} \label{eqn:n-a-N-criteria}
    \begin{cases}
        \mbox{If } n \mbox{ even} &\mbox{then } \Re(a)\geq0 \mbox{ and } N=n/2. \\
        \mbox{If } n \mbox{ odd}  &\mbox{then } \Re(a)=0;
        \begin{cases}
            \mbox{if } a=\ri &\mbox{then } N = (n+1)/2, \\
            \mbox{if } a=-\ri &\mbox{then } N = (n-1)/2,
        \end{cases}
    \end{cases}
\end{equation}
the initial datum $q_0$ has sufficient decay, and the boundary data $h_k$, boundary coefficients $\M{b}{k}{j}$ and initial datum are sufficiently smooth and compatible in the sense that
\begin{equation} \label{eqn:CompatCond}
    \sum_{j=0}^{n-1} \M{b}{k}{j}(0) q^{(j)}_0(0) = h_k(0).
\end{equation}

Before studying the general theory of such IdBVP, we present the full solution for four examples of IdBVP for typical linear evolution equations: the heat equation, linear free space Schr\"{o}dinger equation (LS) and the linearised Korteweg de Vries equation (LKdV).

\begin{prob*}[Heat equation with a homogeneous dynamic Robin condition]
    We seek $q$ for which
    \begin{subequations} \label{eqn:HeatIdBVP}
    \begin{align}
        \label{eqn:HeatIdBVP.PDE} \tag{\theparentequation.PDE}
        \left[ \partial_t - \partial_x^2 \right] q(x,t) &= 0 & (x,t) &\in (0,\infty)\times(0,T), \\
        \label{eqn:HeatIdBVP.IC} \tag{\theparentequation.IC}
        q(x,0) &= q_0(x) & x &\in [0,\infty), \\
        \label{eqn:HeatIdBVP.dBC} \tag{\theparentequation.dBC}
        b(t) q(0,t) + q_x(0,t) &= 0, & t &\in [0,T],
    \end{align}
    \end{subequations}
    where $q_0$ and $b$ are sufficiently smooth and $b(0)q_0(0)+q'_0(0)=0$.
    Here $n=2$, $a=1$ and $N=1$.
\end{prob*}

\begin{prob*}[LS with an inhomogeneous dynamic Robin condition]
    We seek $q$ for which
    \begin{subequations} \label{eqn:LSIdBVP}
    \begin{align}
        \label{eqn:LSIdBVP.PDE} \tag{\theparentequation.PDE}
        \left[ \partial_t - \ri\partial_x^2 \right] q(x,t) &= 0 & (x,t) &\in (0,\infty)\times(0,T), \\
        \label{eqn:LSIdBVP.IC} \tag{\theparentequation.IC}
        q(x,0) &= q_0(x) & x &\in [0,\infty), \\
        \label{eqn:LSIdBVP.dBC} \tag{\theparentequation.dBC}
        b(t) q(0,t) + q_x(0,t) &= h(t), & t &\in [0,T],
    \end{align}
    \end{subequations}
    where $q_0,b,h$ are sufficiently smooth and $b(0)q_0(0)+q'_0(0)=h(0)$.
    Here $n=2$, $a=\ri$ and $N=1$.
\end{prob*}

\begin{prob*}[LKdV with one dynamic boundary condition]
    We seek $q$ for which
    \begin{subequations} \label{eqn:LKdV1IdBVP}
    \begin{align}
        \label{eqn:LKdV1IdBVP.PDE} \tag{\theparentequation.PDE}
        \left[ \partial_t + \partial_x^3 \right] q(x,t) &= 0 & (x,t) &\in (0,\infty)\times(0,T), \\
        \label{eqn:LKdV1IdBVP.IC} \tag{\theparentequation.IC}
        q(x,0) &= q_0(x) & x &\in [0,\infty), \\
        \label{eqn:LKdV1IdBVP.dBC} \tag{\theparentequation.dBC}
        b(t) q(0,t) + q_{xx}(0,t) &= 0, & t &\in [0,T],
    \end{align}
    \end{subequations}
    where $q_0$ and $b$ are sufficiently smooth and $b(0)q_0(0)+q''_0(0)=0$.
    Here $n=3$, $a=-\ri$ and $N=1$.
\end{prob*}

\begin{prob*}[LKdV with two dynamic boundary conditions]
    We seek $q$ for which
    \begin{subequations} \label{eqn:LKdV2IdBVP}
    \begin{align}
        \label{eqn:LKdV2IdBVP.PDE} \tag{\theparentequation.PDE}
        \left[ \partial_t - \partial_x^3 \right] q(x,t) &= 0 & (x,t) &\in (0,\infty)\times(0,T), \\
        \label{eqn:LKdV2IdBVP.IC} \tag{\theparentequation.IC}
        q(x,0) &= q_0(x) & x &\in [0,\infty), \\
        \label{eqn:LKdV2IdBVP.dBC1} \tag{\theparentequation.dBC1}
        b(t) q(0,t) + q_x(0,t) &= 0, & t &\in [0,T], \\
        \label{eqn:LKdV2IdBVP.dBC2} \tag{\theparentequation.dBC2}
        \beta(t) q(0,t) + q_{xx}(0,t) &= 0, & t &\in [0,T],
    \end{align}
    \end{subequations}
    where $q_0,b,\beta$ are sufficiently smooth, $b(0)q_0(0)+q'_0(0)=0$ and $\beta(0)q_0(0)+q''_0(0)=0$.
    Here $n=3$, $a=\ri$ and $N=2$.
\end{prob*}

Problems such as~\eqref{eqn:HeatIdBVP} for the heat equation have received attention from the applied mathematics, physics and engineering communities, due to their applications, including geophysics~\cite{BBHMWW1983a}, nuclear reactor design~\cite{Hol1972a}, and Newtonian cooling in a bath of corrosive liquid~\cite{OM1974a,WYYZ2017a}.

In early works on these problems~\cite{IS1965a,IS1966a,Pos1970a}, the dynamic Robin boundary condition is replaced with a Neumann boundary condition by means of a change of variables.
This has the regrettable side effect of also introducing nonlinear terms to the partial differential equation.
The nonlinear terms are typically assumed small, and the resulting linearization is solved, though quantified justification for ignoring the nonlinear terms is rather complex.

One cannot reasonably expect classical spectral methods, based on separation of variables, to be successful for such problems, because the spectrum of the spatial differential operator necessarily depends upon time; it is unsurprising that some kind of approximate method is usually applied in place of a fully analytic approach.
Nevertheless, a finite interval version of problem~\eqref{eqn:HeatIdBVP} was famously studied by~\cite{OM1974a}, where there was an attempt at separation of variables, modified by the admission that in such a problem the eigenvalues and eigenfunctions necessarily depend upon time.
Unfortunately, the procedure reduced to an infinite system of ordinary differential equations, whose simultaneous solution is not clearly attainable.
Nor was the solution found for any approximation beyond the first order.

Other proposed solution methods for problems similar to~\eqref{eqn:HeatIdBVP} include the Laplace transform and finite difference approaches of~\cite{Koz1970a,BBHMWW1983a}, reduction to an integral equation which admits numerical analysis, as in~\cite{BSY2012a}, or reduction via shifting functions to a variable coefficient partial differential equation with static boundary conditions followed by series expansion, as per~\cite{CSHL2010a,LT2015a} and many other works citing these.
Related nonlinear problems, whose general analytic solution is unfeasible, are studied in~\cite{Moi2008a}, where Lie group techniques are used to find some classes of invariant solutions.

We aim not only to derive a solution representation, but also to verify analytically that our claimed solution representation describes a valid solution of the original problem.
The shifting functions approach of~\cite{LT2015a} has the advantage of being fully analytic but, because the variable coefficient partial differential equation is not in the standard Sturm Liouville form, the ``try functions'' (basis functions) used in the series expansions are not eigenfunctions of the differential operator at any particular time, much less at all times, necessitating a complicated $(x,t)$ dependence in the solution formula.
Equipped only with such a solution formula, it would be difficult to establish well posedness or convergence results beyond the verification of accuracy for fixed numbers of terms typically presented in such works.
Moreover, as demonstrated in~\cite{Jac1915a}, the completeness of the try functions cannot be expected to generalise to odd order problems such as~\eqref{eqn:LKdV1IdBVP} or~\eqref{eqn:LKdV2IdBVP}.
The latter weakness is also suffered by the approach of~\cite{OM1974a}.

Attempting to solve problems such as these via a temporal Laplace transform yields expressions involving root functions of the spectral variable.
When the problem is of higher spatial order and lower order terms also appear in the PDE, inverses of polynomial functions are required.
This can make inversion of the Laplace transform difficult.
A temporal transform also requires assumption that the solution exists for all positive time and a control on the solution for large time, both of which are unnatural assumptions for evolution equations such as those we study here.
Therefore, we avoid methods, such as those of~\cite{Koz1970a,BBHMWW1983a}, based on a temporal Laplace transform.
The approach of~\cite{BSY2012a} is to formulate a map, encoded as an integral equation, from the dynamic boundary condition to a Dirichlet boundary value.
Our approach begins in a similar spirit: though through different means, we also derive an integral equation for such a \emph{Data to uNknown map} (D to N map).
However, taking advantage of some analytic techniques, we are able to reexpress our integral equation in an alternative form for simpler analysis: a variable coefficient fractional linear ordinary differential equation, which we analyse via the fractional Frobenius method.

In this work, we present an extension to IdBVP of class~\eqref{eqn:IdBVP} of the Fokas transform method~\cite{Fok2008a,DTV2014a,FP2015a}, also known as the unified transform method.
This method began as an inverse scattering technique suitable for solving integrable nonlinear evolution equations on domains with a boundary~\cite{Fok2002b,FIS2005a}.
It was soon observed that the method also provided new results for linear evolution equations, particularly those of high spatial order with difficult boundary conditions.
It has been used to establish well posedness and obtain solution representations for IBVP with arbitrary static boundary forms on the half line and finite interval~\cite{FP2001a,Fok2002a,Smi2012a}, and to analyse problems with conditions specified at an interface~\cite{DPS2014a,DS2015a,SS2015a,DSS2016a,DS2020a,STV2019a}, multipoint conditions~\cite{FM2013a,PS2018a}, fully nonlocal conditions~\cite{MS2018a}, and problems in two spatial dimensions~\cite{Fok2002a,FM2015a}.
It has also provided new insights into the spectral theory of spatial differential operators~\cite{Pel2004a,PS2013a,Smi2015a,FS2016a,PS2016a,ABS2020a}, including many with boundary conditions not open to study under the Birkhoff regularity framework.
The method has been extended to admit elliptic equations on nonseparable domains in 2 or 3 dimensions~\cite{Fok2001a,FS2012a,Ash2012a,FK2014a,Cro2015a,CL2018b}, where there has been significant recent progress on numerics~\cite{Col2020a,CFF2018a,CFH2019a}.
Nonlinear elliptic equations have also been studied~\cite{PP2010a,FL2011a,FLP2013a}.

We give an overview of the Fokas transform method in \S\ref{sec:UTM}, both its application to half line IBVP for linear evolution equations with static boundary forms and the necessary extensions for IdBVP~\eqref{eqn:IdBVP}.
In an IBVP, the D to N map reduces to a linear system, solvable by Cramer's rule, but the corresponding reduction for IdBVP does not yield an equivalent linear system.
Instead, we derive a variable coefficient fractional linear ordinary differential equation, or system thereof.

Fortunately for our purposes, there exists an established Frobenius method for variable coefficient fractional linear ordinary differential equations.
We use the approach described in~\cite[\S7.5]{KST2006a} for equations involving (sequential) Caputo fractional derivatives however, for simplicity, we consider only ordinary points of the fractional differential equations we study.
When working with fractional derivatives, it is important to select the appropriate fractional derivative for the task.
Because our fractional derivative appears first disguised within a Fourier transform, we could have chosen any of the various established inequivalent definitions that have similar Fourier transforms.
For reasons discussed below, we prefer the Caputo definition of the fractional derivative and, where necessary, use the sequential Caputo derivative, rather then Caputo derivatives of mixed orders.

Recently, there has been some interest in the interaction between the Fokas transform method and fractional derivatives.
In~\cite{BFF2018a}, the method is extended to admit linear evolution equations of fractional spatial order, in terms of the Riemann-Liouville fractional differential operator.
The approach and results of~\cite{BFF2018a} are completely different from ours because problem~\eqref{eqn:IdBVP} features only spatial differential operators of integer order, while in~\cite{BFF2018a} the boundary forms are static.

More cognate of the present work is~\cite{GPV2019a}, in which the authors express the Maxey-Riley equation with Basset history term as a constant coefficient fractional linear ordinary differential equation, then formulate the latter as the Dirichlet to Neumann map for the heat equation.
Using the Fokas transform method, they solve the D to N map, and thus the Maxey-Riley equation.
Their procedure, expressing a fractional linear ordinary differential equation as a D to N map, is the antithesis of our technique, in which we express a D to N map as a fractional differential equation.
The relative value of each method depends upon the difficulties of solving the problem in its original and transformed modes.
Many of the resultant IBVP in~\cite{GPV2019a} are modified (oblique) Robin problems, so they can be solved using established analytical techniques for the Fokas transform method, although they also present a numerical method to solve the nonlinear modified Robin problems that can occur.
In contrast, there is no known analytical method for solving the IdBVP studied in the present work.
Therefore, we prefer the improved efficiency that may be afforded by an approximation method in one dimension than two; we prefer to reduce our IdBVP in space and time first to their D to N maps and then to fractional linear ordinary differential equations in time only.

\subsection*{Layout of paper}

In \S\ref{sec:UTM} we give a brief overview of the Fokas transform method as applied to IBVP with static boundary forms for evolution equations on the half line, then describe the extension to admit the dynamic boundary forms of present interest.
This section also includes a general implementation of the parts of the Fokas transform method that remain unchanged in the new context, summarised in proposition~\ref{prop:SolRep2valid}.
We conclude \S\ref{sec:UTM} with a lemma that will not be used until the following stage of the method, but is simple enough to prove in general.

In \S\ref{sec:EgHeat} we present the solution of example problem~\eqref{eqn:HeatIdBVP}.
We continue with solutions for problems~\eqref{eqn:LSIdBVP}--\eqref{eqn:LKdV2IdBVP} in \S\ref{sec:EgLSLKdV}.
For these sections, we eschew the general approach of \S\ref{sec:UTM} in order to present simplified arguments appropriate for some examples.
There is some discussion of each approach and its applicability beyond the problem considered, with the fullest discussion in \S\ref{sec:EgHeat}.

The arguments already presented for specific examples are generalised in \S\ref{sec:General} to the full class of IdBVP~\eqref{eqn:IdBVP}.
We conclude with some remarks on further generalisations in \S\ref{sec:Conclusion}.
For convenience, we also provide the brief appendical \S\ref{sec:CaputoFourier} on the interaction of the Fourier transform with the Caputo fractional differential operators of interest.

\section{Half line Fokas transform method} \label{sec:UTM}

\subsection{Overview for IBVP and extensions for IdBVP}

The Fokas transform method may be seen as an extension of the Fourier transform method to a much wider class of boundary value probelms.
Alternatively, it may be seen as a method for the derivation of a Fourier transform pair tailored to a particular IBVP.
Here, we outline the former view.
As described in~\cite{MS2018a}, the method consists of three stages.

In stage~1, one assumes that there exists a solution of the PDE satisfying the initial condition, and derives two equations that must be satisfied by that solution: the global relation and another equation known as the Ehrenpreis form.
The global relation is a linear equation linking integral transforms of values of the solution, and various derivatives, on the boundaries of the space time domain.
The Ehrenpreis form appears close to providing a representation of the solution, but it depends upon all of the boundary values, approximately twice as many as the maximum number that could be explicitly specified by boundary conditions in a well posed problem.
The first stage of the method does not depend upon the boundary conditions in any way, so it is not sensitive to whether we study an IBVP or IdBVP.
We implement stage~1 of the Fokas transform method for evolution equations on the half line in \S\ref{sec:UTM.Stage1}.

In stage~2 of the Fokas transform method, one continues under the assumptions of stage~1, but assumes further that $q$ satisfies the boundary conditions, and derives a \emph{Data to uNknown map} (D to N map) so that all boundary values can be represented in terms of only the boundary and initial data of the problem.
The main tool in this construction is the global relation, but it also requires some asymptotic and Jordan's lemma arguments to show that certain terms, which are not specified in terms of the data, do not contribute to the solution representation.
An attractive feature of the Fokas transform method for static IBVP is that the D to N map may be formulated entirely in spectral space, and as a linear system of dimension no greater than the spatial order of the PDE.
Unfortunately, the argument is complicated by dynamic boundary forms.
For problem~\eqref{eqn:IdBVP}, it is no longer possible to express the D to N map as a linear system, nor even to solve it entirely in spectral space.
However, by applying a Fourier transform to the global relation, we are able to reduce the D to N map to a variable coefficient sequential Caputo fractional linear ordinary differential equation, or a system of such equations.
In \S\S\ref{sec:EgHeat}--\ref{sec:General}, we describe this reformulation of the D to N map and, where the established general theory of such objects permits, solve the resultant equations.
Substituting the solution into the Ehrenpreis form yields an effective representation of the solution to the IdBVP.
Because we solve the fractional differential equation using an iterative method, this statement requires some justification.
We provide the full argument in theorem~\ref{thm:Heat.Convergence} for IdBVP~\eqref{eqn:HeatIdBVP}, theorem~\ref{thm:LS.Convergence} for IdBVP~\eqref{eqn:LSIdBVP}, theorem~\ref{thm:LKdV1.Convergence} for IdBVP~\eqref{eqn:LKdV1IdBVP}, theorem~\ref{thm:LKdV2.Convergence} for IdBVP~\eqref{eqn:LKdV2IdBVP}, and theorem~\ref{thm:General.convergence} for the general IdBVP~\eqref{eqn:IdBVP}.

In stages~1 and~2, we work under the assumption of existence of a solution and, eventually, obtain an explicit formula for that solution.
Therefore, we also prove unicity of this solution, under the assumption of its existence.
Stage~3 of the Fokas transform method justifies that existence assumption.
Typically, for a static IBVP, one treats the solution representation obtained in stage~2 as an ansatz of irrelevant providence, and shows directly that a function defined by that formula satisfies the original IBVP.
The existence result then bootstraps the whole argument.
In the present work, for efficincy of presentation, we separate the arguments of stage~3 into two parts.
The proofs that the PDE and initial condition hold are presented in general in proposition~\ref{prop:SolRep2valid}.
Satisfaction of the dynamic boundary conditions is established in theorems~\ref{thm:HeatStage3} for IdBVP~\eqref{eqn:HeatIdBVP} and~\ref{thm:General.Stage3} for the general IdBVP~\eqref{eqn:IdBVP}.

\subsection{Stage 1} \label{sec:UTM.Stage1}

Let the Schwartz space $\mathcal{S}[0,\infty)$ be the space of half line restrictions of smooth functions decaying rapidly along with all their derivatives.
Consider how the differential operator $(-i\, \mathrm{d} / \mathrm{d} x)^n:\mathcal{S}[0,\infty)\to\mathcal{S}[0,\infty)$ interacts with the half line complex Fourier transform
\[
    \hat{\phi}(\la) = \int_0^\infty\re^{-\ri\la x}\phi(x)\D x.
\]
Integration by parts $n$ times yields
\begin{equation}
    \reallywidehat{\left(-\ri \frac{\D}{\D x}\right)^n\phi}(\la) = \la^n\hat\phi(\la) - (-\ri)^n\sum_{j=0}^{n-1} (\ri \la)^{n-1-j} \phi^{(j)}(0).
\end{equation}

Suppose that $q:[0,\infty)\times[0,T]$ satisfies
\begin{subequations} \label{eqn:SpaceForq}
\begin{align}
    \forall\, t &\in[0,T], & q(\argdot,t) &\in \mathcal{S}[0,\infty), \\
    \forall\, x &\in[0,\infty),\,j\in\{0,1,\ldots,n-1\}, & \partial_x^jq(x,\argdot) &\in \mathrm{AC}[0,T],
\end{align}
\end{subequations}
in which $\mathrm{AC}[0,T]$ represents the space of functions absolutely continuous on the closed real interval $[0,T]$.
We apply the half line complex Fourier transform to equation~\eqref{eqn:IdBVP.PDE} for a temporal ordinary differential equation in $\hat{q}(\la;\argdot)$.
Solving the differential equation with initial condition the half line complex Fourier transform of equation~\eqref{eqn:IdBVP.IC}, we obtain the \emph{global relation}
\begin{equation} \label{eqn:GR}
    \hat{q}_0(\la) - \re^{a\la^nt}\hat{q}(\la;t) = \sum_{j=0}^{n-1} c_j(\la) f_j(\la;t),
\end{equation}
valid for all $t\in[0,T]$ and all $\la\in\clos(\CC^-)$,
where
\begin{align}
    c_j(\la) &= \frac{-a\la^n}{(\ri\la)^{j+1}}, \\
    f_j(\la;t) &= \int_0^t \re^{a\la^ns} \partial_x^jq(0,s) \D s.
\end{align}

Rearranging and applying an inverse Fourier transform to global relation~\eqref{eqn:GR}, we obtain solution representation
\begin{equation} \label{eqn:SolRep1}
    2\pi q(x,t) = \int_{-\infty}^\infty \re^{\ri\la x-a\la^nt} \hat{q}_0(\la) \D\la - \int_{-\infty}^\infty \re^{\ri\la x-a\la^nt} \sum_{j=0}^{n-1} c_j(\la) f_j(\la;t) \D\la.
\end{equation}

For any $R\geq0$, we define the sectorial domains
\begin{align}
    D_R &= \{ \la \in\CC^+ : \lvert\la\rvert>R \mbox{ and } \Re(a\la^n)<0 \}, \\
    E_R &= \CC^+ \setminus \clos(D_R).
\end{align}
Suppose that $\partial_x^jq(0,s)$ and its (temporal) derivative are both $L_1$ functions on $[0,T]$.
Integrating by parts, and applying the Riemann-Lebesgue lemma in the second statement,
\begin{align*}
    \left\lvert \re^{-a\la^nt}c_j(\la)f_j(\la) \right\rvert
    &= \left\lvert \frac{c_j(\la)}{a\la^n} \right\rvert \times \left\lvert \partial_x^jq(0,t) - \re^{-a\la^nt}\partial_x^jq(0,0) - \re^{-a\la^nt} \int_0^t \re^{a\la^ns}\frac{\D}{\D s}\left[ \partial_x^jq(0,s) \right] \D s \right\rvert \\
    &= \mathcal{O}\left(\lvert\la\rvert^{-(j+1)}\right),
\end{align*}
uniformly in $t$ and $\arg(\la)$, as $\la\to\infty$ within $\clos(E_R)$.
Therefore, by Jordan's lemma, for all $x>0$ and all $R\geq0$,
\begin{equation} \label{eqn:cdef1}
    \int_{\partial E_R} \re^{\ri\la x-a\la^nt} c_j(\la)f_j(\la;t)\D\la=0.
\end{equation}
A very similar argument (see the proof of proposition~\ref{prop:SolRep2valid}.\ref{itm:SolRep2valid.tauR} for the argument), but this time analysing limits within $\clos(D_R)$, yields that, for any $\tau\in[t,T]$,
\begin{equation} \label{eqn:cdef2}
    \int_{\partial D_R} \re^{\ri\la x-a\la^nt} c_j(\la)[f_j(\la;\tau)-f_j(\la;t)]\D\la=0.
\end{equation}
Applying equations~\eqref{eqn:cdef1} and~\eqref{eqn:cdef2} to equation~\eqref{eqn:SolRep1}, we find that, for all $x>0$ and all $t\in[0,T]$,
\begin{equation} \label{eqn:SolRep2}
    2\pi q(x,t) = \int_{-\infty}^\infty \re^{\ri\la x-a\la^nt} \hat{q}_0(\la) \D\la - \int_{\partial D_R} \re^{\ri\la x-a\la^nt} \sum_{j=0}^{n-1} c_j(\la) f_j(\la;\tau) \D\la,
\end{equation}
in which the constants $R\geq0$ and $\tau\in[t,T]$ are arbitrary.
Equation~\eqref{eqn:SolRep2} is often called the Ehrenpreis form.

\subsection{Validity of the solution representation}

\begin{ntn}
    The maps $\partial_x^jq(0,\argdot) \mapsto f_j(\argdot;t)$ are instances of the transform
    \begin{equation} \label{eqn:transformdefn}
        F[\phi](\la;t) = \int_0^t \re^{a\la^ns} \phi(s) \D s,
    \end{equation}
    in which $\phi = \partial_x^j q(0,\argdot)$.
    We introduce this $F$ notation because transform~\eqref{eqn:transformdefn} will be applied to a more general class of functions than just the boundary values.
\end{ntn}

We have shown that, under the assumption that there exists a function $q$ satisfying equations~\eqref{eqn:IdBVP.PDE}--\eqref{eqn:IdBVP.IC}, it must be that $q$ satisfies equation~\eqref{eqn:SolRep2}.
The following theorem establishes the reverse of this statement: if equation~\eqref{eqn:SolRep2} is taken as the definition of a function $q$, then $q$ also satisfies equations~\eqref{eqn:IdBVP.PDE}--\eqref{eqn:IdBVP.IC}.
This is part of stage~3 of the Fokas transform method.
Because we have not yet implemented the D to N map, nor considered the dynamic boundary condition~\eqref{eqn:IdBVP.dBC} at all, we are not yet ready to complete all of stage~3.
Nevertheless, because this part of the argument is relatively unchanged from the method for static IBVP, we present it early.

\begin{prop} \label{prop:SolRep2valid}
    Suppose that $q_0\in\mathcal{S}[0,\infty)$ and, for each $j\in\{0,1,\ldots,n-1\}$, the function $\phi_j\in C^\infty[0,T]$.
    Suppose further that these functions satisfy the compatibility conditions
    \[
        q_0^{(j)}(0) = \phi_j(0), \qquad j \in \{0,1,\ldots,n-1\}.
    \]
    Let $\Omega=[0,\infty)\times[0,T]$, and define $q:\Omega\to\CC$ by
    \begin{equation} \label{thm:SolRep2valid.defnq}
        q(x,t) = \frac{1}{2\pi} \int_{-\infty}^\infty \re^{\ri\la x-a\la^nt} \hat{q}_0(\la) \D\la - \frac{1}{2\pi} \int_{\partial D_R} \re^{\ri\la x-a\la^nt} \sum_{j=0}^{n-1} c_j(\la) F[\phi_j](\la;\tau) \D\la,
    \end{equation}
    in which the integrals should be understood as the joint principal value
    \[
        \lim_{M\to\infty} \left[
            \int_{-M}^M \ldots \D\la + \int_{\partial D_R} \chi_{B(0,M)}(\la) \ldots \D\la,
        \right]
    \]
    for $\chi_{B(0,M)}$ the indicator function of the ball centred at $0$ with radius $M$.
    Then
    \begin{enumerate}
        \item{
            \label{itm:SolRep2valid.Convergence}
            In equation~\eqref{thm:SolRep2valid.defnq}, both integrands may be multiplied by any polynomial in $\la$ of degree no greater than $n-1$, and the resulting integral converges uniformly on $\Omega$.
            If the polynomial is of degree $n$ then the resulting integral converges pointwise on $\Omega$ and locally uniformly in the interior of $\Omega$.
        }
        \item{
            \label{itm:SolRep2valid.tauR}
            In equation~\eqref{thm:SolRep2valid.defnq}, $R\geq0$ and $\tau\in[t,T]$ may be freely chosen without affecting the definition of $q$.
        }
        \item{
            \label{itm:SolRep2valid.PDE}
            $q$ satisfies partial differential equation~\eqref{eqn:IdBVP.PDE}.
        }
        \item{
            \label{itm:SolRep2valid.IC}
            $q$ satisfies initial condition~\eqref{eqn:IdBVP.IC}.
        }
    \end{enumerate}
\end{prop}

\begin{rmk}
    This proof has appeared before (see, for example,~\cite{Fok2008a}), so we do not give the full argument here.
    Much less regularity and compatibility of the data is necessary~\cite{BT2019a}, but the smoothness simplifies the argument.
\end{rmk}

\begin{proof}[Sketch proof of proposition~\ref{prop:SolRep2valid}.\ref{itm:SolRep2valid.Convergence}]
    We study expressions of the form
    \begin{equation}
        q(x,t) = \frac{1}{2\pi} \int_{-\infty}^\infty \re^{\ri\la x-a\la^nt} \la^r \hat{q}_0(\la) \D\la - \int_{\partial D_R} \re^{\ri\la x-a\la^nt} \la^r \sum_{j=0}^{n-1} c_j(\la) F[\phi_j](\la;\tau) \D\la,
    \end{equation}
    for $r\in\{0,1,\ldots,n\}$.
    Changing variables $\la=\re^{\ri\theta}\sqrt[n]{-\ri\rho}$, $\la^n=\ri\rho/a$ for appropriate choices of $\theta$ on each connected component of $\partial D_R$, the latter integral lies on the real line, except near $0$, where it is deformed slightly into the upper half plane.
    From here, the inner integrals $\hat{q}_0$ and $F[\phi_j]$ may be integrated by parts an appropriate number of times to cancel the $\la^r$ blow up; the resulting boundary terms cancel by the compatibility conditions.
    If $r<n$, a Dirichlet test justifies uniform convergence.
    If $r=n$ then, the compatibility conditions exhausted, one appeals to the uniform convergence theorem of~\cite{Sta1973a}.
\end{proof}

\begin{rmk}
    It is a corollary of proposition~\ref{prop:SolRep2valid}.\ref{itm:SolRep2valid.Convergence} that $q$ achieves its $t\to0^+$ initial time limit and $q$ and its first $n-1$ spatial partial derivatives achieve their $x\to0^+$ boundary values.
\end{rmk}

\begin{proof}[Proof of proposition~\ref{prop:SolRep2valid}.\ref{itm:SolRep2valid.tauR}]
    Considering each $\phi_j$ as a compactly supported absolutely integrable function on the real line, its full line Fourier transform is an entire function.
    But $F[\phi_j](\la)$ is nothing more than a composition of a power function with the Fourier transform.
    Hence each integrand in equation~\eqref{thm:SolRep2valid.defnq} is entire.
    Therefore, the latter integral may be deformed over any finite region without affecting its value; the solution representation is independent of $R$.
    
    Integrating by parts, and applying the Riemann-Lebesgue lemma in the second statement,
    \begin{multline*}
        \left\lvert \re^{-a\la^nt} c_j(\la) \big(F[\phi_j](\la;\tau) - F[\phi_j](\la;t)\big) \right\rvert \\
        = \left\lvert \frac{c_j(\la)}{a\la^n} \right\rvert \times \left\lvert \re^{a\la^n(\tau-t)}\phi_j(\tau) - \phi_j(t) - \re^{-a\la^nt} \int_t^\tau \re^{a\la^ns}\phi'_j(s) \D s \right\rvert \\
        = \mathcal{O}\left(\lvert\la\rvert^{-(j+1)}\right),
    \end{multline*}
    uniformly in $t$ and $\arg(\la)$, as $\la\to\infty$ within $\clos(D_R)$.
    Therefore, by Jordan's lemma, for all $x>0$ and all $R\geq0$,
    \begin{equation*}
        \int_{\partial D_R} \re^{\ri\la x-a\la^nt} c_j(\la) \big(F[\phi_j](\la;\tau) - F[\phi_j](\la;t)\big) \D\la = 0.
    \end{equation*}
    Freedom of $\tau\in[t,T]$ follows.
\end{proof}

\begin{proof}[Proof of proposition~\ref{prop:SolRep2valid}.\ref{itm:SolRep2valid.PDE}]
    By proposition~\ref{prop:SolRep2valid}.\ref{itm:SolRep2valid.Convergence}, the relevant partial derivatives exist and are given by differentiating under the integrals.
    Partial differential equation~\eqref{eqn:IdBVP.PDE} follows immediately from the simple $(x,t)$ dependence of the integrands.
\end{proof}

\begin{proof}[Proof of proposition~\ref{prop:SolRep2valid}.\ref{itm:SolRep2valid.IC}]
    By proposition~\ref{prop:SolRep2valid}.\ref{itm:SolRep2valid.tauR}, we may select $R=0$ and $\tau=0$.
    The latter integral evaluates to $0$.
    Initial condition~\eqref{eqn:IdBVP.IC} now follows by the usual Fourier inversion theorem.
\end{proof}

The following proposition, proved in~\cite{FS1999a}, gives a characterisation of the kinds of sets of functions $\phi_j$ that can be chosen to ensure that the boundary values of $q$ are exactly the appropriate $\phi_j$.

\begin{prop} \label{prop:BVattained}
    Suppose all criteria of proposition~\eqref{prop:SolRep2valid} and that, for each $t\in[0,T]$, there exists some function $\gamma\in L_1[0,\infty)$ for which
    \begin{equation}
        \hat{q}_0(\la) - \re^{a\la^nt}\hat{\gamma}(\la) = \sum_{j=0}^{n-1} c_j(\la) F[\phi_j](\la).
    \end{equation}
    Then $\partial_x^jq(0,t)=\phi_j(t)$.
\end{prop}

\subsection{A useful application of Jordan's lemma}

The following lemma is very similar to one of the main tools in stage~2 of the Fokas transform method for IBVP with static boundary forms.
Its application, as presented in \S\S\ref{sec:EgHeat}--\ref{sec:General}, for IdBVP is a little different, which permits a slight simplification from the usual form of the lemma.
In either case, it is used to remove terms involving Fourier transforms of $q(\argdot;T)$ from various equations.

\begin{lem} \label{lem:RemoveqT}
    Suppose $\phi$ and its derivative are $L_1$ integrable functions on $[0,\infty)$.
    If $2\leq n\in\NN$, $T>t$, and $\frac{-(2n-1)\pi}{2n} \leq \theta \leq \frac{-\pi}{2n}$, then
    \[
        \int_{-\infty}^\infty \re^{\ri\rho(T-t)} \hat{\phi}\left(\re^{\ri\theta}\sqrt[n]{-\ri\rho}\right) \D \rho = 0.
    \]
\end{lem}

\begin{proof}
    Because $\sqrt[n]{\argdot}$ is the principle branch of the $n$th root, the function $\sqrt[n]{-\ri\rho}$ has a branch cut only for $\rho\in-\ri[0,\infty)$ and is analytic elsewhere.
    Suppose $0 \leq \arg(\rho) \leq \pi$.
    Then $\frac{-\pi}{2n} \leq \arg(\sqrt[n]{-\ri\rho}) \leq \frac{\pi}{2n}$, so $-\pi \leq \arg\left(\re^{\ri\theta}\sqrt[n]{-\ri\rho}\right) \leq 0$.
    Therefore, the inner integral
    \[
        \hat{\phi}\left(\re^{\ri\theta}\sqrt[n]{-\ri\rho}\right) = \int_0^\infty \re^{-\ri\re^{\ri\theta}\sqrt[n]{-\ri\rho}x}\phi(x)\D x
    \]
    is continuous and analytic on the closed upper half plane, except at $0$ where it has a continuous branch point.
    Moreover, integration by parts and the Riemann-Lebesgue lemma imply that, as $\rho\to\infty$ from within $\clos(\CC^+)$,
    \[
        \left\lvert \hat{\phi}\left(\re^{\ri\theta}\sqrt[n]{-\ri\rho}\right) \right\rvert = \mathcal{O}\left(\lvert\rho\rvert^{-1/n}\right),
    \]
    uniformly in $\arg(\rho)$.
    The result follows by Jordan's lemma.
\end{proof}


\section{Example: heat equation} \label{sec:EgHeat}

For IdBVP~\eqref{eqn:HeatIdBVP}, the global relation~\eqref{eqn:GR} at final time $T$ simplifies to
\begin{equation} \label{eqn:Heat.GR}
    \hat{q}_0(\la) - \re^{\la^2T}\hat{q}(\la;T) = \ri\la f_0(\la;T) + f_1(\la;T),
\end{equation}
and Ehrenpreis form~\eqref{eqn:SolRep2} becomes
\begin{equation} \label{eqn:Heat.SolRep1}
    2\pi q(x,t) = \int_{-\infty}^\infty \re^{\ri\la x-\la^2t} \hat{q}_0(\la) \D\la - \int_{\partial D_R} \re^{\ri\la x-\la^2t} \left[ \ri\la f_0(\la;\tau) + f_1(\la;\tau) \right] \D\la.
\end{equation}
Applying transform~\eqref{eqn:transformdefn} to dynamic boundary condition~\eqref{eqn:HeatIdBVP.dBC},
we obtain, for all $t\in[0,T]$ and all $\la\in \CC$,
\begin{equation} \label{eqn:Heat.dBCtransformed}
    f_1(\la;t) = - F[b(\argdot)q(0,\argdot)](\la;t),
\end{equation}
so we can rewrite the solution representation as
\begin{equation} \label{eqn:Heat.SolRep2}
    2\pi q(x,t) = \int_{-\infty}^\infty \re^{\ri\la x-\la^2t} \hat{q}_0(\la) \D\la - \int_{\partial D_R} \re^{\ri\la x-\la^2t} F[(\ri\la-b(\argdot))q(0,\argdot)](\la;T) \D\la.
\end{equation}
Equation~\eqref{eqn:Heat.dBCtransformed} also allows us to simplify the global relation~\eqref{eqn:Heat.GR} to
\begin{equation} \label{eqn:Heat.GRSimplified}
    \hat{q}_0(\la) - \re^{\la^2T}\hat{q}(\la;T) = F[(\ri\la-b(\argdot))q(0,\argdot)](\la;T),
\end{equation}
valid for all $\la\in\clos(\CC^-)$.

\begin{figure}
    \centering
    \includegraphics{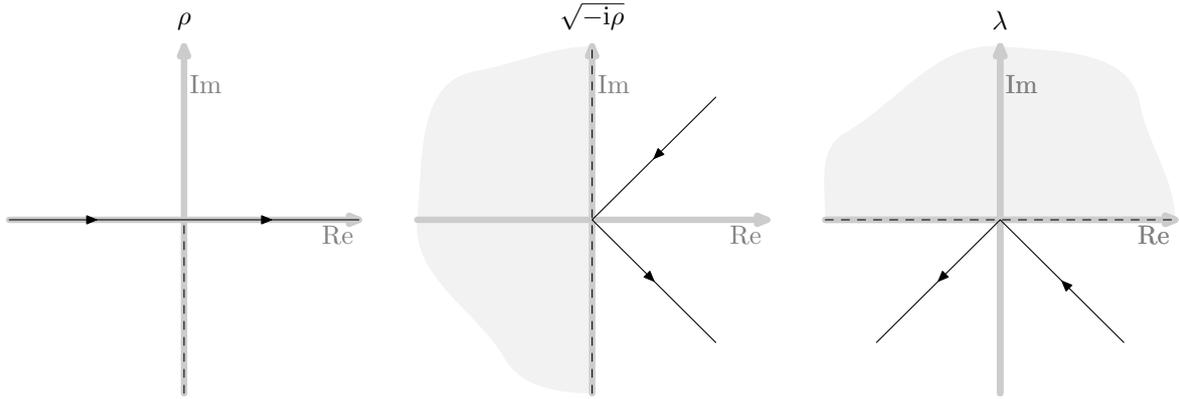}
    \caption{
        The map $\rho\mapsto\la=-\ri\sqrt{-\ri\rho}$ is biholomorphic on the unshaded region.
        The function defining the map has a branch cut along the negative imaginary half axis.
    }
    \label{fig:Heat.maps}
\end{figure}

Denoting by $\sqrt{\argdot}$ the principle branch of the square root function, we make the change of variables $\la=-\ri\sqrt{-\ri\rho}$, $\la^2=\ri\rho$, as depicted in figure~\ref{fig:Heat.maps}, in equation~\eqref{eqn:Heat.GRSimplified} to obtain
\begin{align} \notag
    \hat{q}_0\left(-\ri\sqrt{-\ri\rho}\right) - \re^{\ri\rho T}\hat{q}\left(-\ri\sqrt{-\ri\rho};T\right) &= F\left[\left(\sqrt{-\ri\rho}-b(\argdot)\right)q(0,\argdot)\right]\left(-\ri\sqrt{-\ri\rho};T\right) \\
    &= \int_0^T \re^{\ri\rho s} \left(\sqrt{-\ri\rho}-b(s)\right)q(0,s) \D s,
    \label{eqn:Heat.GRChanged}
\end{align}
valid for all $\rho\in\CC$, with a cut along $-\ri[0,\infty)$, and a square root branch point at $0$.
Note that the right of this equation is a temporal Fourier transform (under definition~\ref{defn:FTforCaputo}) of the function $\left(\sqrt{-\ri\rho}-b(\argdot)\right)q(0,\argdot)$ supported on $[0,T]$.

The multiplier $\sqrt{-\ri\rho}$ of a Fourier transform type integral, appearing on the right of equation~\eqref{eqn:Heat.GRChanged}, is redolent of the formula for Fourier transforms of derivatives, except that this appears to represent a half derivative.
Indeed, by proposition~\ref{prop:FTCaputoFiniteInterval},
\begin{multline} \label{eqn:Heat.GRChanged2}
    \hat{q}_0\left(-\ri\sqrt{-\ri\rho}\right) - \re^{\ri\rho T}\hat{q}\left(-\ri\sqrt{-\ri\rho};T\right) \\
    = \int_0^T \re^{\ri\rho s} \left[ \Caputo{1/2}{q(0,\argdot)}(s) - b(s)q(0,s) \right] \D s + \frac{1}{\sqrt{-\ri\rho}}\left( q(0,0) - q(0,T)\re^{\ri\rho T} \right),
\end{multline}
for $\prescript{\ensuremath{\mathrm{C}}}{}{\ensuremath{\mathrm{D}}}_{0\Mspacer+}^{1/2}$ the Caputo half derivative operator of definition~\ref{defn:Caputo}.
Let $\Gamma_R$ be the contour following the real line in the increasing sense, except perturbed away from $0$ into $\CC^+$ along a circular arc of radius $R$.
Then, multiplying by $\re^{-\ri\rho t}$ and integrating each side of equation~\eqref{eqn:Heat.GRChanged2} in $\rho$ along $\Gamma_R$, it is immediate from Jordan's lemma that the final term on the right side evaluates to $0$.
Taking the limit as $R\to0$, and applying the Fourier inversion theorem for the Fourier transform of definition~\ref{defn:FTforCaputo}, we obtain
\begin{equation} \label{eqn:Heat.FLODEqT}
    \Caputo{1/2}{y}(t) - b(t)y(t) = \frac{1}{2\pi} \int_{-\infty}^\infty \re^{-\ri\rho t} \left[ \hat{q}_0\left(-\ri\sqrt{-\ri\rho}\right) + \frac{q_0(0)}{\sqrt{-\ri\rho}} - \re^{\ri\rho T}\hat{q}\left(-\ri\sqrt{-\ri\rho};T\right) \right] \D\rho,
\end{equation}
where $y(t) = q(0,t)$.
By lemma~\ref{lem:RemoveqT} with $\theta=-\pi/2$, equation~\eqref{eqn:Heat.FLODEqT} simplifies to
\begin{equation} \label{eqn:Heat.FLODE}
    \Caputo{1/2}{y}(t) - b(t)y(t) = \frac{1}{2\pi} \int_{-\infty}^\infty \re^{-\ri\rho t} \left[ \hat{q}_0\left(-\ri\sqrt{-\ri\rho}\right) + \frac{q_0(0)}{\sqrt{-\ri\rho}} \right] \D\rho = : g(t),
\end{equation}
in which the new datum $g$ has been defined as this unusual double Fourier transform of the initial datum.

Equation~\eqref{eqn:Heat.FLODE} is an inhomogeneous variable coefficient Caputo fractional linear ordinary differential equation.
The theory of such equations is presented in~\cite[\S7.5]{KST2006a} and~\cite{Use2008a}.
We proceed under the assumption that both data $g$ and $b$ are $\frac{1}{2}$ analytic about the $\frac{1}{2}$ ordinary point $0$ with radius $R$, so that
\begin{equation} \label{eqn:Heat.defnBG}
    b(t) = \sum_{u=0}^\infty B_u t^{u/2}, \qquad \qquad g(t) = \sum_{u=0}^\infty G_u t^{u/2}.
\end{equation}
In practice, these coefficients may be calculated by successive application of the Caputo fractional derivative operator~\cite{Use2008a}.
The existence of a solution $y$, $\frac{1}{2}$ analytic at $0$ with radius $R$, is guaranteed by~\cite[theorem~7.16]{KST2006a}. So we seek a series solution of the form
\begin{subequations} \label{eqn:Heat.FLODE.Solution}
\begin{equation}
    q(0,t) = y(t) = \sum_{u=0}^\infty Y_u t^{u/2}.
\end{equation}
By~\cite[remark~7.2]{KST2006a},
\[
    \Caputo{1/2}{y}(t) = \sum_{u=1}^\infty Y_u \frac{\Gamma(u/2+1)}{\Gamma((u-1)/2+1)} t^{(u-1)/2}.
\]
Substituting into equation~\eqref{eqn:Heat.FLODE}, we obtain the recurrence relation
\begin{equation} \label{eqn:Heat.FLODE.Solution.Coeffs}
    Y_{u+1} = \frac{\Gamma((u+2)/2)}{\Gamma((u+3)/2)} \left( G_u + \sum_{v=0}^u Y_vB_{u-v} \right),
\end{equation}
\end{subequations}
and $Y_0=q_0(0)$ because $\lim_{t\to0} q(0,t)=q_0(0)$.

The following theorem establishes that equations~\eqref{eqn:Heat.SolRep2} and~\eqref{eqn:Heat.FLODE.Solution} provide a valid solution to IdBVP~\eqref{eqn:HeatIdBVP}, and the solution representation obtained is effective.

\begin{thm} \label{thm:Heat.Convergence}
    Suppose that $q_0\in\mathcal{S}[0,\infty)$.
    Suppose further that $b$ and $g$, the latter given by the definition on the right of equations~\eqref{eqn:Heat.FLODE}, are $\frac{1}{2}$ analytic functions at $0$, with $\frac{1}{2}$ power series given by equations~\eqref{eqn:Heat.defnBG}, with radii of convergence both strictly greater than $T$.
    For $U\in\NN$, let
    \[
        y_U(t) = \sum_{u=0}^U Y_u t^{u/2},
    \]
    in which $Y_0=q_0(0)$ and, for all integer $u\geq0$, $Y_{u+1}$ is given by relation~\eqref{eqn:Heat.FLODE.Solution.Coeffs}.
    For $U\in\NN$ and $(x,t)\in[0,\infty)\times[0,T]$, let
    \[
        q_U(x,t) = \frac{1}{2\pi} \int_{-\infty}^\infty \re^{\ri\la x-\la^2t} \hat{q}_0(\la) \D\la - \frac{1}{2\pi} \int_{\partial D_R} \re^{\ri\la x-\la^2t} F[(\ri\la-b(\argdot))y_U(\argdot)](\la;T) \D\la.
    \]
    
    Then
    \begin{enumerate}
        \item{
            For each $U\in\NN$, the formula defining $q_U(x,t)$ is given by integrals that converge uniformly in $(x,t)\in[0,\infty)\times[0,T]$; differentiating under the integral yields a formula for $\partial_xq_U(x,t)$, also in terms of integrals uniformly convergent in $(x,t)\in[0,\infty)\times[0,T]$.
        }
        \item{
            For all $U\in\NN$, $q_U$ satisfies equation~\eqref{eqn:HeatIdBVP.PDE}.
        }
        \item{
            For all $U\in\NN$, $q_U$ satisfies equation~\eqref{eqn:HeatIdBVP.IC}.
        }
        \item{
            If $q$ is a solution of IdBVP~\eqref{eqn:HeatIdBVP} then $q$ is unique among functions satisfying~\eqref{eqn:SpaceForq} and the criteria of this theorem, and, uniformly in $(x,t)\in[0,\infty)\times[0,T]$,
            \[
                \lim_{U\to\infty} q_U(x,t) = q(x,t).
            \]
        }
    \end{enumerate}
\end{thm}

\begin{proof}
    The first three statements follow immediately from proposition~\ref{prop:SolRep2valid} for this particular IdBVP.
    
    As argued above, if $q$ solves IdBVP~\eqref{eqn:HeatIdBVP} then $q(0,t)$ is a solution of equation~\eqref{eqn:Heat.FLODE} with initial condition $q(0,0)=q_0(0)$.
    By~\cite[theorem~7.16]{KST2006a}, this problem has a unique $\frac{1}{2}$ analytic solution.
    This justifies the claimed unicity of $q$, but also it must be that
    \[
        q(0,t) = y_U(t) + R_U(t), \qquad \mbox{where} \qquad R_U(t) = \sum_{u=U+1}^\infty Y_u t^{u/2};
    \]
    because $R_U$ corresponds to the remainder of the $\frac{1}{2}$ power series for the same $\frac{1}{2}$ analytic function $y$, we refer to $R_U$ as the \emph{$U$ tail} corresponding to partial sum $y_U$.
    Moreover, because $(R_U)_{U\in\NN}$ is the sequence of tails of a $\frac{1}{2}$ power series with radius of convergence strictly greater than $T$, $R_U(t)$ converges to $0$ uniformly in $t\in[0,T]$.
    
    The linear dependence of $q_U$ on $y_U$, and the fact that $q(0,\argdot)$ must be the exact solution of equation~\eqref{eqn:Heat.FLODE}, imply that $q_U(x,t) = q(x,t) - E_U(x,t)$ in which
    \[
        2\pi E_U(x,t) = \int_{\partial D_R} \re^{\ri\la x- \la^2t} \int_0^T \re^{\la^2s}(\ri\la-b(s))R_U(s) \D s \D\la.
    \]
    Therefore, the desired result is equivalent to $E_U(x,t)\to0$ as $U\to\infty$, uniformly in $(x,t)\in[0,\infty)\times[0,T]$.
    
    We make the change of variables $\la=\ri\sqrt{-\ri\rho}$, $\la^2=\ri\rho$ to find
    \begin{equation} \label{eqn:Heat.ConvergenceProof.1}
        2\pi E_U(x,t) = \int_{-\infty}^\infty \re^{-\sqrt{-\ri\rho}x-\ri\rho t} \int_0^T \re^{\ri\rho s} \mathcal{R}_U(s) \D s \D\rho,
    \end{equation}
    where we applied proposition~\ref{prop:FTRLFracInt} to obtain the simple formula
    \begin{equation} \label{eqn:Heat.ConvergenceProof.defnmathcalR}
        \mathcal{R}_U(t) = \frac{-1}{2}\left(R_U(t) + \left(\Msup{I}{0}{+}{1/2}(bR_U)\right)(t)\right),
    \end{equation}
    in which $\ensuremath{\mathrm{I}}_{0\Mspacer+}^{1/2}$ is the Riemann-Liouville half integral operator of definition~\ref{defn:Caputo}.
    Multiplying $R_U$ by another $\frac{1}{2}$ analytic function $b$ yields another function converging uniformly in $t\in[0,T]$, and~\cite[theorem~2.7]{Die2010a} guarantees that the uniformity is preserved when a Riemann-Liouville fractional integral operator is applied.
    Hence $\mathcal{R}_U(t)\to0$ as $U\to\infty$, uniformly in $t\in[0,T]$.
    
    In the case $x=0$, the Fourier inversion theorem simplifies equation~\eqref{eqn:Heat.ConvergenceProof.1} to $E_U(0,t)=\mathcal{R}_U(t)$, so $E_U(0,t)\to0$ as $U\to\infty$, uniformly in $t\in[0,T]$.
    
    Now consider an arbitrary fixed $x>0$.
    By equation~\eqref{eqn:Heat.ConvergenceProof.1},
    \[
        E_U(x,t) = (\Phi(\argdot;x)*\mathcal{R}_U)(t),
    \]
    in which $*$ represents convolution and the integral in definition
    \[
        \Phi(t;x) = \int_{-\infty}^\infty \re^{-\ri\rho t} \re^{-\sqrt{-\ri\rho}x} \D\rho,
    \]
    is guaranteed to converge by the fact $\Re(-\sqrt{-\ri\rho}x)\to-\infty$ as $\rho\to\pm\infty$.
    Moreover,
    \[
        \left\lVert \Phi(\argdot;x) \right\rVert_1 \leq 2T\int_0^\infty \re^{-\sqrt{\rho/2}x}\D\rho = \frac{8T}{x^2}.
    \]
    Therefore
    \[
        \left\lvert E_U(x,t) \right\rvert \leq \frac{8T}{x^2} \left\lVert \mathcal{R}_U \right\rVert_\infty.
    \]
    For any fixed $x_0>0$, the right side converges to $0$ as $U\to\infty$, uniformly in $(x,t)\in[x_0,\infty)\times[0,T]$.
    
    Applying proposition~\ref{prop:SolRep2valid}.\ref{itm:SolRep2valid.Convergence} with $q_0=0$, $\phi_0=R_U$ and $\phi_1=-bR_U$, compatible because $R_U(0)=0$, the integral on the right of equation~\eqref{eqn:Heat.ConvergenceProof.1} has $x$ derivative given by differentiation under the integral, and this holds even at $x=0$.
    Therefore
    \[
        \partial_x E_U(0,t) = \int_{-\infty}^\infty -\sqrt{-\ri\rho} \re^{-\ri\rho t} \int_0^T \re^{\ri\rho s} \mathcal{R}_U(s)\D s \D\rho = -\Caputo{1/2}{\mathcal{R}_U}(t),
    \]
    in which the boundary terms vanish by Jordan's lemma and because $\mathcal{R}_U(0)=0$.
    For $U>1$, $\mathcal{R}_U$ is differentiable, with derivative also the tail of a $\frac{1}{2}$ power series with the same radius of convergence, so $\mathcal{R}'_U(t)\to0$ uniformly in $t\in[0,T]$, and, using again~\cite[theorem~2.7]{Die2010a}, so does $\Caputo{1/2}{\mathcal{R}_U}(t)$.
    In particular, this means that $E_U(x,t)$ is equicontinuous in $t\in[0,T]$ and $U\in\NN$ at $x=0$.
    Finally, we conclude that $E_U(x,t)\to0$ as $U\to\infty$, uniformly in $(x,t)\in[0,\infty)\times[0,T]$.
\end{proof}

\begin{rmk}
    Although we assumed strong regularity of $q_0$, this is not strictly necessary.
    However, the requirement of $\frac{1}{2}$ analyticity on $b,g$ suggests diminishing returns in reducing the regularity of $q_0$.
    There already exists a theory for the Frobenius method at $\frac{1}{2}$ singular points~\cite{KST2006a}, but its application in this setting is left for future work.
\end{rmk}

\begin{rmk}
    The uniform convergence of $q_U(x,t)\to q(x,t)$ in theorem~\ref{thm:Heat.Convergence} is at the same rate as the fractional power series tail $\mathcal{R}_U$ converges to $0$, as is most clearly demonstrated in the above proof at $x=0$, but is also evident at $x_0>0$.
    Indeed, even convergence pointwise in $x$ is only as fast as the fractional power series converges.
\end{rmk}

This completes stage~2 of the Fokas transform method for IdBVP~\eqref{eqn:HeatIdBVP}.
However, theorem~\ref{thm:Heat.Convergence} relied upon the assumption that there exists a solution of IdBVP~\eqref{eqn:HeatIdBVP}, which has yet to be justified.
The following theorem shows that the function $q_U$, as constructed in theorem~\ref{thm:Heat.Convergence}, does indeed converge to a solution of IdBVP~\eqref{eqn:HeatIdBVP}.
Therefore, such a solution must exist and stage~3 of the Fokas transform method is complete.

\begin{thm} \label{thm:HeatStage3}
    There exists a solution to IdBVP~\eqref{eqn:HeatIdBVP}.
    Moreover, supposing all criteria of theorem~\ref{thm:Heat.Convergence} except existence of a solution for IdBVP~\eqref{eqn:HeatIdBVP}, then, uniformly in $t\in[0,T]$,
    \[
        \lim_{U\to\infty} \left( b(t) q_U(0,t) + \partial_x q_U(0,t) \right) = 0,
    \]
    so that, in the limit $U\to0$, dynamic boundary condition~\eqref{eqn:HeatIdBVP.dBC} is recovered.
\end{thm}

\begin{proof}
    Let $y$ be the solution of fractional differential equation~\eqref{eqn:Heat.FLODE}, whose unicity is guaranteed by~\cite[theorem~7.16]{KST2006a}.
    Then
    \[
        y(t) = y_U(t) + R_U(t),
    \]
    for $R_U$ the $U$ tail of the fractional power series for $y$.
    We define
    \[
        q(x,t) = \frac{1}{2\pi} \int_{-\infty}^\infty \re^{\ri\la x-\la^2t} \hat{q}_0(\la) \D\la - \frac{1}{2\pi} \int_{\partial D_R} \re^{\ri\la x-\la^2t} F[(\ri\la-b(\argdot))y(\argdot)](\la;T) \D\la.
    \]
    Thus $q$ is defined by equation~\eqref{thm:SolRep2valid.defnq} in which $\phi_0(t)=y(t)$ and $\phi_1(t)=-b(t)y(t)$.
    Because $y$ satisfies equation~\eqref{eqn:Heat.FLODE}, there is, for each $t\in[0,T]$, some function $\gamma(x,t)$ for which
    \[
        \hat{q}_0(\la) - \re^{\la^2t}\hat{\gamma}(\la;t) = F[(\ri\la-b(\argdot))y(\argdot)](\la;t) = \ri\la F[\phi_0] + F[\phi_1](\la;t).
    \]
    Therefore, by proposition~\ref{prop:BVattained},
    \[
        b(t)q(0,t) + q_x(0,t) = b(t)\phi_0(t) + \phi_1(t) = b(t)y(t)-b(t)y(t) = 0,
    \]
    so $q$ satisfies dynamic boundary condition~\eqref{eqn:HeatIdBVP.dBC}.
    
    By construction, $q$ depends upon $y$ in exactly the same way $q_U$ depends upon $y_U$, and that dependence is linear.
    Therefore, $q_U(x,t) = q(x,t) - E_U(x,t)$, in which
    \[
        2\pi E_U(x,t) = \int_{\partial D_R} \re^{\ri\la x-\la^2t} \int_0^T \re^{\la^2s} \left( \ri\la - b(s) \right) R_U(s) \D s \D \la.
    \]
    So it is equivalent for us to prove that, uniformly in $t\in[0,T]$,
    \[
        \lim_{U\to\infty} \mathcal{E}_U(t) = 0,
    \]
    in which $\mathcal{E}_U(t) = 2\pi[b(t) E_U(0,t) + \partial_x E_U(0,t)]$.

    Making the change of variables $\la=\ri\sqrt{-\ri\rho}$, $\la^2=\ri\rho$,
    \[
        \mathcal{E}_U(t) = \frac{-1}{2}\int_{-\infty}^\infty \re^{-\ri\rho t} \frac{b(t)-\sqrt{-\ri\rho}}{\sqrt{-\ri\rho}} \int_0^T \re^{\ri\rho s} \left(b(s)-\sqrt{-\ri\rho}\right)R_U(s) \D s \D \rho.
    \]
    Multiplying out the parentheses to obtain four terms, each expressed as its own double integral, we notice that each term is a pseudodifferential operator applied to $R_U$.
    Two are trivial to evaluate and turn out to be the same; the other two represent a Riemann-Liouville $\frac{1}{2}$ integral and a Caputo $\frac{1}{2}$ derivative of the remainder function $R_U$, so
    \begin{equation*}
        \left\lvert \mathcal{E}_U(t) \right\rvert
        \leq \frac{1}{2} \left( 2\left\lvert b(t) \right\rvert \times \left\lvert R_U(t) \right\rvert + \left\lvert b(t) \right\rvert \times \left\lvert \left(\Msup{I}{0}{+}{1/2} (R_U b) \right)(t) \right\rvert + \left\lvert \left(\Msup{I}{0}{+}{1/2} R'_U \right)(t) \right\rvert \right).
    \end{equation*}
    Because $T$ is strictly less than the radius of convergence of the $\frac{1}{2}$ power series, $R_U(t)\to0$ as $U\to\infty$, uniformly in $t\in[0,T]$, and so does $R_U(t)b(t)$, the product of $R_U$ with another $\frac{1}{2}$ analytic function.
    Similarly, $R'_U$ also converges to $0$ uniformly on $[0,T]$.
    Therefore, by~\cite[theorem~2.7]{Die2010a}, so do their Riemann-Liouville $\frac{1}{2}$ integrals.
    As $b$ is $\frac{1}{2}$ analytic, it is certainly bounded on $[0,T]$.
    Therefore, uniformly in $t\in[0,T]$, $\lim_{U\to\infty} \mathcal{E}_U(t) = 0$.
\end{proof}

\subsection{Caputo versus Riemann-Liouville fractional derivatives} \label{ssec:HeatEg.CaputovsRL}
One may ask why we choose the Caputo fractional derivative, rather than the Riemann-Liouville fractional derivative, in the argument between equations~\eqref{eqn:Heat.GRChanged} and~\eqref{eqn:Heat.GRChanged2}.
Indeed, with a trivial adaptation of the proof we give, one may easily prove a version of proposition~\ref{prop:FTCaputoFiniteInterval} for Riemann-Liouville fractional derivatives with the same $\sqrt{-\ri\rho}$ multiplying the Fourier transform:
\[
    \left(\mathcal{F}\prescript{\ensuremath{\mathrm{RL}}}{}{\ensuremath{\mathrm{D}}}_{0\Mspacer+}^{\alpha}y\right)(\rho) = (-\ri\rho)^\alpha \left(\mathcal{F}y\right)(\rho) - \left( \Msup{I}{0}{+}{\alpha} y \right)(0) + \left( \Msup{I}{0}{+}{\alpha} y \right)(T)\re^{\ri\rho T}.
\]
Therefore, by inserting the above boundary terms instead, we could obtain a version of equation~\eqref{eqn:Heat.GRChanged2} with the Riemann-Liouville fractional derivative instead of the Caputo fractional derivative.
However, the lack of decay of the Riemann-Liouville boundary terms would preclude their subsequent elimination via Jordan's lemma.
Therefore, the Caputo fractional derivative is the better choice.

\begin{rmk}
    Note that the boundary terms on the right of equation~\eqref{eqn:Heat.GRChanged2} are related to the two terms on the left of that equation.
    Indeed, integrating by parts in each of the Fourier transforms on the left, one obtains leading order terms which exactly cancel with the terms on the right.
    This cancellation is also unique to the Caputo fractional derivative.
\end{rmk}

\section{Further examples} \label{sec:EgLSLKdV}

\subsection{Linear Schr\"{o}dinger with inhomogeneous dynamic boundary condition} \label{ssec:Eg.LS}

In IdBVP~\eqref{eqn:LSIdBVP}, the global relation~\eqref{eqn:GR} at final time is
\[
    \hat{q}_0(\la) - \re^{\ri\la^2T}\hat{q}(\la;T) = \ri\left(\ri\la f_0(\la;T)+f_1(\la;T)\right).
\]
Applying transform~\eqref{eqn:transformdefn} to dynamic boundary condition~\eqref{eqn:LSIdBVP.dBC}, the global relation simplifies further to
\[
    \hat{q}_0(\la) - \re^{\ri\la^2T}\hat{q}(\la;T) = \ri F\left[\left(\ri\la-b(\argdot)\right) q(0,\argdot)\right](\la;T) + \ri F[h](\la;t).
\]
We make the change of variables $\la = -\sqrt{\ri}\sqrt{-\ri\rho}$, $\la^2=\rho$ and find that
\begin{multline*}
    \frac{1}{\sqrt{\ri}} \hat{q}_0\left(-\sqrt{\ri}\sqrt{-\ri\rho}\right) - \frac{1}{\sqrt{\ri}} \re^{\ri\rho T} \hat{q}\left(-\sqrt{\ri}\sqrt{-\ri\rho};T\right) \\
    = \int_0^T\re^{\ri\rho s}\left(\sqrt{-\ri\rho}-\sqrt{\ri}b(s)\right)q(0,s)\D s + \sqrt{\ri} F[h]\left(-\sqrt{\ri}\sqrt{-\ri\rho};T\right).
\end{multline*}
We use proposition~\ref{prop:FTCaputoFiniteInterval} to substitute
\[
    \int_0^T \re^{\ri\rho s} \sqrt{-\ri\rho}q(0,s)\D s
\]
for the Fourier transform of a Caputo fractional derivative, with associated boundary terms.
Just as in \S\ref{sec:EgHeat}, we multiply both sides of the above equation by $\re^{-\ri\rho t}$ and integrate along contour $\Gamma_R$ in $\rho$, where $\Gamma_R$ is the real line perturbed into $\CC^+$ along a semicircular arc centered at $0$.
Jordan's lemma removes the $\hat{q}(\argdot,T)$ and $q(0,T)$ terms.
A limit as $R\to0$ and an application of the Fourier inversion theorem for the Fourier transform of definition~\ref{defn:FTforCaputo}, yields
\begin{multline} \label{eqn:LS.FLODE}
    \Caputo{1/2}{y}(t) - \sqrt{\ri}b(t)y(t) \\
    = \frac{-\sqrt{\ri}}{2\pi}\int_{-\infty}^\infty \re^{-\ri\rho t} \left[ \ri \hat{q}_0\left(-\sqrt{\ri}\sqrt{-\ri\rho}\right) + \frac{q_0(0)}{\sqrt{\ri}\sqrt{-\ri\rho}} + F[h]\left(-\sqrt{\ri}\sqrt{-\ri\rho};T\right) \right] \D\rho
    =: \sqrt{\ri} g(t),
\end{multline}
where $y=q(0,\argdot)$.
Note that, because $h$ and $q_0$ are data of the IdBVP, we have once again reduced the D to N map to an inhomogeneous variable coefficient Caputo fractional linear ordinary differential equation.

\begin{thm} \label{thm:LS.Convergence}
    Suppose that $q_0\in\mathcal{S}[0,\infty)$.
    Suppose further that $b$ and $g$, the latter given by the definition on the right of equation~\eqref{eqn:LS.FLODE}, are $\frac{1}{2}$ analytic functions at $0$, with $\frac{1}{2}$ power series
    \begin{equation} \label{eqn:LS.defnBG}
        b(t) = \sum_{u=0}^\infty B_u t^{u/2}, \qquad \qquad g(t) = \sum_{u=0}^\infty G_u t^{u/2},
    \end{equation}
    with radii of convergence both strictly greater than $T$.
    For $U\in\NN$, let
    \[
        y_U(t) = \sum_{u=0}^U Y_u t^{u/2},
    \]
    in which $Y_0=q_0(0)$ and, for all integer $u\geq0$, $Y_{u+1}$ is given by
    \begin{equation} \label{eqn:LS.FLODE.Solution.Coeffs}
        Y_{u+1} = \sqrt{\ri}\frac{\Gamma((u+2)/2)}{\Gamma((u+3)/2)} \left( G_u + \sum_{v=0}^u Y_vB_{u-v} \right).
    \end{equation}
    For $U\in\NN$ and $(x,t)\in[0,\infty)\times[0,T]$, let
    \[
        q_U(x,t) = \frac{1}{2\pi} \int_{-\infty}^\infty \re^{\ri\la x-\ri\la^2t} \hat{q}_0(\la) \D\la - \frac{\ri}{2\pi} \int_{\partial D_R} \re^{\ri\la x-\ri\la^2t} F[(\ri\la-b(\argdot))y_U(\argdot)+h(\argdot)](\la;T) \D\la.
    \]
    
    Then
    \begin{enumerate}
        \item{
            For each $U\in\NN$, the formula defining $q_U(x,t)$ is given by integrals that converge uniformly in $(x,t)\in[0,\infty)\times[0,T]$; differentiating under the integral yields a formula for $\partial_xq_U(x,t)$, also in terms of integrals uniformly convergent in $(x,t)\in[0,\infty)\times[0,T]$.
        }
        \item{
            For all $U\in\NN$, $q_U$ satisfies equation~\eqref{eqn:LSIdBVP.PDE}.
        }
        \item{
            For all $U\in\NN$, $q_U$ satisfies equation~\eqref{eqn:LSIdBVP.IC}.
        }
        \item{
            If $q$ is a solution of IdBVP~\eqref{eqn:LSIdBVP} then $q$ is unique among functions satisfying~\eqref{eqn:SpaceForq} and the criteria of this theorem, and, uniformly in $t\in[0,T]$ and pointwise in $x\in[0,\infty)$,
            \[
                \lim_{U\to\infty} q_U(x,t) = q(x,t).
            \]
        }
        \item{
            IdBVP~\eqref{eqn:LSIdBVP} has a solution $q$ given by
            \[
                q(x,t) = \frac{1}{2\pi} \int_{-\infty}^\infty \re^{\ri\la x-\ri\la^2t} \hat{q}_0(\la) \D\la - \frac{\ri}{2\pi} \int_{\partial D_R} \re^{\ri\la x-\ri\la^2t} F[(\ri\la-b(\argdot))y(\argdot)+h(\argdot)](\la;T) \D\la.
            \]
            in which $y$ is the solution of equation~\eqref{eqn:LS.FLODE}.
            Dynamic boundary condition~\eqref{eqn:LSIdBVP.dBC} is satisfied in the limit $U\to\infty$, uniformly in $t\in[0,T]$.
        }
    \end{enumerate}
\end{thm}

\begin{proof}
    The proof may be adapted from the proof of theorem~\ref{thm:Heat.Convergence} with reference to the derivation in~\S\ref{ssec:Eg.LS}.
    
    Alternatively, we proceed as follows.
    The first three claims are proposition~\ref{prop:SolRep2valid}.
    The existence and unicity of a solution for Caputo fractional differential equation~\eqref{eqn:LS.FLODE} is established in~\cite[theorem~7.16]{KST2006a}.
    The fourth claim is now a corollary of theorem~\ref{thm:General.convergence}.
    The fifth claim follows from theorem~\ref{thm:General.Stage3}.
\end{proof}

\subsection{LKdV with one dynamic boundary condition} \label{ssec:Eg.LKdV1}

IdBVP~\eqref{eqn:LKdV1IdBVP}, has global relation~\eqref{eqn:GR}
\begin{equation} \label{eqn:LKdV1.GR}
    \hat{q}_0(\la) - \re^{-\ri\la^3T}\hat{q}(\la;T) = \la^2 f_0(\la;T) - \ri\la f_1(\la;T) - f_2(\la;T).
\end{equation}
Upon application of transform~\eqref{eqn:transformdefn} to dynamic boundary condition~\eqref{eqn:LKdV1IdBVP.dBC}, the global relation simplifies to
\begin{equation} \label{eqn:LKdV1.GRSimplified}
    \hat{q}_0(\la) - \re^{-\ri\la^3T}\hat{q}(\la;T) = F\left[\left(\la^2 + b(\argdot)\right) q(0,\argdot)\right](\la;T) - \ri\la f_1(\la;T).
\end{equation}

Even after this reduction, global relation~\eqref{eqn:LKdV1.GRSimplified} is an integral equation relating two unknown quantities: $q(0,\argdot)$ and $q_x(0,\argdot)$.
Observe that this integral equation holds for all $\la\in\clos\CC^-$ and, when applying the inverse Fourier transform, we require validity on only a single complex contour.
We will find two contours on which to pose an inverse Fourier transform, thereby obtaining a system of two fractional linear ordinary differential equations in the two unknowns.
Recalling that it will be crucial to establish that any terms involving $\hat{q}(\la;T)$ evaluate to zero after an appropriate inverse Fourier transform has been applied, we investigate the domain of validity of lemma~\ref{lem:RemoveqT}.

We seek two maps satisfying $\la = \re^{\ri\theta}\sqrt[3]{-\ri\rho}$ and $\la^3=-\rho$ for which $\re^{\ri3\theta}=-\ri=\re^{-\ri\pi/2}$ and $-5\pi/6\leq\theta\leq-\pi/6$.
A simple calculation reveals that the only valid choices are $\theta_1=-\pi/6$ and $\theta_2=-5\pi/6$.
Applying those maps, we find
\begin{multline} \label{eqn:LKdV1.GRSimplified2}
    \hat{q}_0\left( \re^{\ri\theta_k}\sqrt[3]{-\ri\rho} \right) - \re^{\ri\rho T} \hat{q}\left( \re^{\ri\theta_k}\sqrt[3]{-\ri\rho} ; T \right) \\
    = F \left[ \left( \re^{\ri2\theta_k}\left(\sqrt[3]{-\ri\rho}\right)^2 + b(\argdot) \right)q(0,\argdot) \right] \left(\re^{\ri\theta_k}\sqrt[3]{-\ri\rho};T\right) - \ri \re^{\ri\theta_k}\sqrt[3]{-\ri\rho} f_1\left(\re^{\ri\theta_k}\sqrt[3]{-\ri\rho};T\right).
\end{multline}
We use proposition~\ref{prop:FTCaputoFiniteInterval} to substitute two of these terms for Caputo fractional derivatives, obtaining
\begin{multline} \label{eqn:LKdV1.GRSimplified3}
    \hat{q}_0\left( \re^{\ri\theta_k}\sqrt[3]{-\ri\rho} \right) - \re^{\ri\rho T} \hat{q}\left( \re^{\ri\theta_k}\sqrt[3]{-\ri\rho} ; T \right) \\
    = \re^{\ri2\theta_k} \left(\mathcal{F} \prescript{\ensuremath{\mathrm{C}}}{}{\ensuremath{\mathrm{D}}}_{0\Mspacer+}^{2/3} y \right)(\rho)
    + \re^{\ri2\theta_k}\left( \frac{y(0)-y(T)\re^{\ri\rho T}}{\sqrt[3]{-\ri\rho}} \right)
    + \left(\mathcal{F}[by]\right)(\rho) \\
    - \ri\re^{\ri\theta_k} \left(\mathcal{F} \prescript{\ensuremath{\mathrm{C}}}{}{\ensuremath{\mathrm{D}}}_{0\Mspacer+}^{1/3} z \right)(\rho)
    - \ri\re^{\ri\theta_k}\left( \frac{z(0)-z(T)\re^{\ri\rho T}}{\left(\sqrt[3]{-\ri\rho}\right)^2} \right),
\end{multline}
in which $y=q(0,\argdot)$ and $z=q_x(0,\argdot)$.
As in the previous examples, we multiply by the Fourier kernel $\re^{-\ri\rho t}$, integrate along $\Gamma_R$ and take the limit as $R\to0$.
The Fourier inversion theorem and lemma~\ref{lem:RemoveqT} provides a system of two equations, one for each $k\in\{1,2\}$,
\begin{multline} \label{eqn:LKdV1.FLODEsystem}
    \re^{\ri2\theta_k} \Caputo{2/3}{y}(t) + b(t)y(t) - \ri\re^{\ri\theta_k} \Caputo{1/3}{z}(t) \\
    = \frac{1}{2\pi} \int_{-\infty}^\infty \re^{-\ri\rho t} \left[ \hat{q}_0\left(\re^{\ri\theta_k}\sqrt[3]{-\ri\rho}\right) - \frac{\re^{\ri2\theta_k}q_0(0)}{\sqrt[3]{-\ri\rho}} + \frac{\ri\re^{\ri\theta_k}q'_0(0)}{\left(\sqrt[3]{-\ri\rho}\right)^2} \right]\D\rho.
\end{multline}
We eliminate $\Caputo{1/3}{z}$ from the system to obtain
\begin{multline} \label{eqn:LKdV1.FLODE}
    \Caputo{2/3}{y}(t) + b(t)y(t) \\
    = \frac{1}{2\sqrt{3}\pi} \int_{-\infty}^\infty \re^{-\ri\rho t} \left[ \re^{-\ri\theta_1} \hat{q}_0\left(\re^{\ri\theta_1}\sqrt[3]{-\ri\rho}\right) - \re^{-\ri\theta_2} \hat{q}_0\left(\re^{\ri\theta_2}\sqrt[3]{-\ri\rho}\right) - \frac{\sqrt{3}q_0(0)}{\sqrt[3]{-\ri\rho}} \right] \D\rho
    =: g(t).
\end{multline}

Solution representation~\eqref{eqn:SolRep2} for this IdBVP is
\begin{equation} \label{eqn:LKdV1.SolRep}
    2\pi q(x,t) = \int_{-\infty}^\infty \re^{\ri\la x+\ri\la^3t} \hat{q}_0(\la) \D\la - \int_{\partial D_R} \re^{\ri\la x+\ri\la^3t} \left[ \la^2 f_0(\la;t) - \ri\la f_1(\la;t) - f_2(\la;t) \right] \D\la,
\end{equation}
and $f_0$ and $f_2$ are readily calculated from $q(0,\argdot)$ and $q_{xx}(0,\argdot)$, respectively, the latter of which can be determined from dynamic boundary condition~\eqref{eqn:LKdV1IdBVP.dBC}.
It remains to find a valid expression for $f_1$ that may be substituted into equation~\eqref{eqn:LKdV1.SolRep}.
There are two methods one may follow.

The first approach is to rearrange one of equations~\eqref{eqn:LKdV1.FLODEsystem} to
\begin{equation}
    \Caputo{1/3}{z}(t)
    = \frac{\ri\re^{-\ri\theta_k}}{2\pi} \int_{-\infty}^\infty \re^{-\ri\rho t} \hat{q}_0\left(\re^{\ri\theta_k}\sqrt[3]{-\ri\rho}\right) \D\rho - \ri\re^{\ri\theta_k} \Caputo{2/3}{y}(t) + b(t)y(t),
\end{equation}
in which $y$ is now known.
Applying the left sided Riemann-Liouville $1/3$ fractional integral operator, we obtain an expression for $z=q_x(0,\argdot)$, and thereby an expression for $f_1$.

The alternative approach, inspired by the Fokas transform method for IBVP, is to return to the original global relation~\eqref{eqn:LKdV1.GR}, apply one of the maps $\la\mapsto\re^{\pm2\ri\pi/3}\la$ and solve for $f_1$ directly.
The resulting expression depends upon $\hat{q}(\argdot;T)$, but a Jordan's lemma argument may be used to show that the integral about $\partial D_R$ of the term involving $\hat{q}(\argdot;T)$ evaluates to $0$.
This is a purely algebraic method for obtaining $f_1$, without going via $q_x(0,\argdot)$, thereby avoiding any more fractional integrals.

Note that both of these methods have some redundancy; one may choose either $k\in\{1,2\}$ in the former, and either of maps $\la\mapsto\re^{\pm2\ri\pi/3}\la$ in the latter.
The second method provides an alternative view of this redundancy.
Indeed, suppose one did not take advantage of the dynamic boundary condition~\eqref{eqn:LKdV1IdBVP.dBC} to find $q_{xx}(0,\argdot)$, so had not yet calculated $f_2$.
One may use both the maps $\la\mapsto\re^{\pm2\ri\pi/3}\la$ in the second approach and solve the resulting system for both $f_1$ and $f_2$.
The same Jordan's lemma type arguments would then establish that all terms featuring $\hat{q}(\argdot;T)$ evaluate to $0$.

\begin{thm} \label{thm:LKdV1.Convergence}
    Suppose that $q_0\in\mathcal{S}[0,\infty)$.
    Suppose further that $b$ and $g$, the latter given by the definition on the right of equation~\eqref{eqn:LKdV1.FLODE}, are $\frac{2}{3}$ analytic functions at $0$, with $\frac{2}{3}$ power series
    \begin{equation} \label{eqn:LKdV1.defnBG}
        b(t) = \sum_{u=0}^\infty B_u t^{2u/3}, \qquad \qquad g(t) = \sum_{u=0}^\infty G_u t^{2u/3},
    \end{equation}
    with radii of convergence both strictly greater than $T$.
    For $U\in\NN$, let
    \[
        y_U(t) = \sum_{u=0}^U Y_u t^{2u/3},
    \]
    in which $Y_0=q_0(0)$ and, for all integer $u\geq0$, $Y_{u+1}$ is given by
    \begin{equation} \label{eqn:LKdV1.FLODE.Solution.Coeffs}
        Y_{u+1} = \frac{\Gamma((2u+3)/3)}{\Gamma((2u+5)/3)} \left( G_u - \sum_{v=0}^u Y_vB_{u-v} \right).
    \end{equation}
    For $U\in\NN$ and $(x,t)\in[0,\infty)\times[0,T]$, let $\alpha=\re^{2\pi\ri/3}$ and
    \begin{multline*}
        q_U(x,t) = \frac{1}{2\pi} \int_{-\infty}^\infty \re^{\ri\la x+\ri\la^3t} \hat{q}_0(\la) \D\la \\
        - \frac{1}{2\pi} \int_{\partial D_R} \re^{\ri\la x+\ri\la^3t} \left((1-\alpha)F[(\la^2-\alpha^2b(\argdot))y_U(\argdot)](\la;T) + \alpha^2\hat{q}_0(\alpha\la)\right)\D\la.
    \end{multline*}
    
    Then
    \begin{enumerate}
        \item{
            For each $U\in\NN$, the formula defining $q_U(x,t)$ is given by integrals that converge uniformly in $(x,t)\in[0,\infty)\times[0,T]$; differentiating under the integral yields a formula for $\partial_xq_U(x,t)$, also in terms of integrals uniformly convergent in $(x,t)\in[0,\infty)\times[0,T]$.
        }
        \item{
            For all $U\in\NN$, $q_U$ satisfies equation~\eqref{eqn:LKdV1IdBVP.PDE}.
        }
        \item{
            For all $U\in\NN$, $q_U$ satisfies equation~\eqref{eqn:LKdV1IdBVP.IC}.
        }
        \item{
            If $q$ is a solution of IdBVP~\eqref{eqn:LKdV1IdBVP} then $q$ is unique among functions satisfying~\eqref{eqn:SpaceForq} and the criteria of this theorem, and, uniformly in $(x,t)\in[0,\infty)\times[0,T]$,
            \[
                \lim_{U\to\infty} q_U(x,t) = q(x,t).
            \]
        }
        \item{
            IdBVP~\eqref{eqn:LKdV1IdBVP} has a solution $q$ given by
            \begin{multline*}
                q(x,t) = \frac{1}{2\pi} \int_{-\infty}^\infty \re^{\ri\la x+\ri\la^3t} \hat{q}_0(\la) \D\la \\
                - \frac{1}{2\pi} \int_{\partial D_R} \re^{\ri\la x+\ri\la^3t} \left((1-\alpha)F[(\la^2-\alpha^2b(\argdot))y(\argdot)](\la;T) + \alpha^2\hat{q}_0(\alpha\la)\right)\D\la.
            \end{multline*}
            in which $y$ is the solution of equation~\eqref{eqn:LKdV1.FLODE}.
            Dynamic boundary condition~\eqref{eqn:LKdV1IdBVP.dBC} is satisfied in the limit $U\to\infty$, uniformly in $t\in[0,T]$.
        }
    \end{enumerate}
\end{thm}

\begin{proof}
    The proof may be adapted from the proof of theorem~\ref{thm:Heat.Convergence} with reference to the derivation in~\S\ref{ssec:Eg.LKdV1}.
    
    Alternatively, the first three results are proposition~\ref{prop:SolRep2valid}, equation~\eqref{eqn:LKdV1.FLODE} has a unique solution by~\cite[theorem~7.16]{KST2006a}, so the fourth result follows by an application of theorem~\ref{thm:General.convergence}.
    The fifth result is theorem~\ref{thm:General.Stage3}.
\end{proof}

\subsection{LKdV with two dynamic boundary conditions} \label{ssec:Eg.LKdV2}

The global relation~\eqref{eqn:GR} in IdBVP~\eqref{eqn:LKdV2IdBVP} is
\[
    \hat{q}_0(\la) - \re^{\ri\la^3T}\hat{q}(\la;T) = -\la^2 f_0(\la;T) + \ri\la f_1(\la;T) + f_2(\la;T).
\]
Applying transform~\eqref{eqn:transformdefn} to dynamic boundary conditions~\eqref{eqn:LKdV2IdBVP.dBC1} and~\eqref{eqn:LKdV2IdBVP.dBC2}, the global relation simplifies to
\begin{equation} \label{eqn:LKdV2.GRSimplified}
    \hat{q}_0(\la) - \re^{\ri\la^3T}\hat{q}(\la;T) = F\left[\left(-\la^2 - \ri\la b(\argdot) - \beta(\argdot)\right) q(0,\argdot)\right](\la;T).
\end{equation}
In order to apply lemma~\ref{lem:RemoveqT}, we seek a single map $\la=\re^{\ri\theta}\sqrt[3]{-\ri\rho}$, $\la^3=\rho$ so that $\re^{\ri3\theta}=\ri=\re^{\ri\pi/2}$ and $-5\pi/6\leq\theta\leq-\pi/6$.
Proceeding with the only valid choice, $\theta=-\pi/2$, and applying proposition~\ref{prop:FTCaputoFiniteInterval} three times, we find
\begin{align} \notag
    &\hspace{-2em}\hat{q}_0\left(-\ri\sqrt[3]{-\ri\rho}\right) - \re^{\ri\rho T}\hat{q}\left(-\ri\sqrt[3]{-\ri\rho};T\right) \\ \notag
    &= F\left[ \left( \left(\sqrt[3]{-\ri\rho}\right)^2 - \sqrt[3]{-\ri\rho} b(\argdot) - \beta(\argdot)\right) y(\argdot) \right]\left(-\ri\sqrt[3]{-\ri\rho};T\right) \\ \notag
    &= -\left(\mathcal{F}[\beta y]\right)(\rho)
    - \left(\mathcal{F} \prescript{\ensuremath{\mathrm{C}}}{}{\ensuremath{\mathrm{D}}}_{0\Mspacer+}^{1/3} [by] \right)(\rho)
    - \frac{b(0)y(0)-b(T)y(T)\re^{\ri\rho T}}{\left(\sqrt[3]{-\ri\rho}\right)^2} \\
    &\hspace{2em} + \left(\mathcal{F} \prescript{\ensuremath{\mathrm{C}}}{}{\ensuremath{\mathrm{D}}}_{0\Mspacer+}^{1/3\Mspacer2} y \right)(\rho)
    + \frac{\Caputo{1/3}{y}(0)-\Caputo{1/3}{y}(T)\re^{\ri\rho T}}{\left(\sqrt[3]{-\ri\rho}\right)^2}
    + \frac{y(0)-y(T)\re^{\ri\rho T}}{\sqrt[3]{-\ri\rho}},
    \label{eqn:LKdV2.GRSimplified2}
\end{align}
in which $y=q(0,\argdot)$.
By definition~\ref{defn:Caputo}, $\Caputo{1/3}{y}(0)=0$.
The inverse Fourier transform and lemma~\ref{lem:RemoveqT} yields a sequential inhomogeneous variable coefficient Caputo fractional linear ordinary differential equation:
\begin{multline} \label{eqn:LKdV2.FLODE}
    \CaputoSeq{1/3}{2}{y}(t) - \Caputo{1/3}{by}(t) - \beta(t)y(t) \\
    = \frac{1}{2\pi} \int_{-\infty}^\infty \re^{-\ri\rho t} \left[ \hat{q}_0\left(-\ri\sqrt[3]{-\ri\rho}\right)  - \frac{q_0(0)\left(\sqrt[3]{-\ri\rho}-b(0)\right)}{\left(\sqrt[3]{-\ri\rho}\right)^2} \right] \D\rho
    =: g(t).
\end{multline}

\begin{rmk}
    In the derivation of equation~\eqref{eqn:LKdV2.GRSimplified2}, we opted to use the sequential Caputo fractional derivative $\CaputoSeq{1/3}{2}{y}$, but we could have used the $2/3$ Caputo fractional derivative $\Caputo{2/3}{y}$ in its stead.
    Indeed, the latter may appear preferable, as it would generate fewer boundary terms.
    Unfortunately, to the best of the authors' knowledge, there is no existence or uniqueness theorem for Caputo fractional linear ordinary differential equations with variable coefficients and nonsequential fractional derivatives of mixed orders.
    Therefore, we elect to employ sequential Caputo fractional derivatives.
\end{rmk}

\begin{thm} \label{thm:LKdV2.Convergence}
    Suppose that $q_0\in\mathcal{S}[0,\infty)$.
    Suppose further that $b$ and $g$, the latter given by the definition on the right of equation~\eqref{eqn:LKdV2.FLODE}, are $\frac{1}{3}$ analytic functions at $0$, with $\frac{1}{3}$ power series
    \begin{equation} \label{eqn:LKdV2.defnBG}
        b(t) = \sum_{u=0}^\infty B_u t^{u/3}, \qquad \qquad \beta(t) = \sum_{u=0}^\infty \mathcal{B}_u t^{u/3}, \qquad \qquad g(t) = \sum_{u=0}^\infty G_u t^{u/3},
    \end{equation}
    with radii of convergence all strictly greater than $T$.
    For $U\in\NN$, let
    \[
        y_U(t) = \sum_{u=0}^U Y_u t^{u/3},
    \]
    in which $Y_0=q_0(0)$, $Y_1=0$ and, for all integer $u\geq0$, $Y_{u+2}$ is given by
    \begin{equation} \label{eqn:LKdV2.FLODE.Solution.Coeffs}
        Y_{u+2} = \frac{\Gamma((u+4)/3)}{\Gamma((u+5)/3)} \sum_{v=0}^{u+1} Y_vB_{u-v} + \frac{\Gamma((u+3)/3)}{\Gamma((u+5)/3)} \left( G_u + \sum_{v=0}^u Y_v\mathcal{B}_{u-v} \right).
    \end{equation}
    For $U\in\NN$ and $(x,t)\in[0,\infty)\times[0,T]$, let
    \[
        q_U(x,t) = \frac{1}{2\pi} \int_{-\infty}^\infty \re^{\ri\la x-\ri\la^3t} \hat{q}_0(\la) \D\la - \frac{1}{2\pi} \int_{\partial D_R} \re^{\ri\la x-\ri\la^3t} F[(-\la^2-\ri\la b(\argdot)-\beta(\argdot))y_U(\argdot)](\la;T) \D\la.
    \]
    
    Then
    \begin{enumerate}
        \item{
            For each $U\in\NN$, the formula defining $q_U(x,t)$ is given by integrals that converge uniformly in $(x,t)\in[0,\infty)\times[0,T]$; differentiating under the integral yields formulae for $\partial_xq_U(x,t)$ and $\partial_{xx}q_U(x,t)$, also in terms of integrals uniformly convergent in $(x,t)\in[0,\infty)\times[0,T]$.
        }
        \item{
            For all $U\in\NN$, $q_U$ satisfies equation~\eqref{eqn:LKdV2IdBVP.PDE}.
        }
        \item{
            For all $U\in\NN$, $q_U$ satisfies equation~\eqref{eqn:LKdV2IdBVP.IC}.
        }
        \item{
            If $q$ is a solution of IdBVP~\eqref{eqn:LKdV2IdBVP} then $q$ is unique among functions satisfying~\eqref{eqn:SpaceForq} and the criteria of this theorem, and, uniformly in $t\in[0,T]$ and pointwise in $x\in[0,\infty)$,
            \[
                \lim_{U\to\infty} q_U(x,t) = q(x,t).
            \]
        }
        \item{
            IdBVP~\eqref{eqn:LKdV2IdBVP} has a solution $q$ given by
            \[
                q(x,t) = \frac{1}{2\pi} \int_{-\infty}^\infty \re^{\ri\la x-\ri\la^3t} \hat{q}_0(\la) \D\la - \frac{1}{2\pi} \int_{\partial D_R} \re^{\ri\la x-\ri\la^3t} F[(-\la^2-\ri\la b(\argdot)-\beta(\argdot))y(\argdot)](\la;T) \D\la.
            \]
            in which $y$ is the solution of equation~\eqref{eqn:LKdV2.FLODE}.
            Dynamic boundary conditions~\eqref{eqn:LKdV2IdBVP.dBC1} and~\eqref{eqn:LKdV2IdBVP.dBC2} are satisfied in the limit $U\to\infty$, uniformly in $t\in[0,T]$.
        }
    \end{enumerate}
\end{thm}

\begin{proof}
    The proof may be adapted from the proof of theorem~\ref{thm:Heat.Convergence} with reference to the derivation in~\S\ref{ssec:Eg.LKdV2}.
    
    Alternatively, note that the first three statements are proposition~\ref{prop:SolRep2valid}.
    As a lemma for the fourth statement, we claim that equation~\eqref{eqn:LKdV2.FLODE} has a unique solution.
    Indeed, although it is usually difficult to formulate a product rule for Caputo fractional derivatives, the $\frac{1}{3}$ analyticity assumption together with Mertens's theorem simplifies the reduction of equation~\eqref{eqn:LKdV2.FLODE} to one that may be considered under the framework of~\cite[theorem~7.18]{KST2006a}.
    Therefore, the fourth statement follows by an application of theorem~\ref{thm:General.convergence}.
    The fifth statement follows from theorem~\ref{thm:General.Stage3}.
\end{proof}

\section{General D to N map and convergence} \label{sec:General}

\begin{lem} \label{lem:Choicestheta}
    Suppose that $n$, $a$, and $N$ satisfy criterion~\eqref{eqn:n-a-N-criteria}.
    There are exactly $n-N$ values of $\theta$ for which both
    \begin{gather}
        \label{eqn:Choicestheta.1}
        \re^{\ri n\theta} = -1/a, \\
        \label{eqn:Choicestheta.2}
        -(2n-1)\pi/2n \leq \theta \leq -\pi/2n
    \end{gather}
    hold.
\end{lem}

\begin{proof}
    We separate the argument into each of the three cases permissible under criteria~\eqref{eqn:n-a-N-criteria}.
    
    Suppose $n$ is odd and $a=-\ri$ so that $n-N=(n+1)/2$ and criterion~\eqref{eqn:Choicestheta.1} is equivalent to $\re^{\ri n\theta}=\re^{-\ri\pi/2}$.
    There are $n-N$ integers $r$ for which $0\leq r\leq n-N-1$.
    For such integers, $1\leq 4r+1 \leq 2n-1$.
    If, for such $r$, $\theta=-\pi(4r+1)/2n$, then $\theta$ satisfies both criteria~\eqref{eqn:Choicestheta.1} and~\eqref{eqn:Choicestheta.2}.
    Increasing $\theta$ by any noninteger multiple of $4\pi/2n$ would result in the failure of criterion~\eqref{eqn:Choicestheta.1}.
    Choosing any integer $r$ for which $0\leq r\leq n-N-1$ is false would falsify criterion~\eqref{eqn:Choicestheta.2}.
    
    Suppose instead that $n$ is odd but $a=\ri$, in which case $n-N=(n-1)/2$ and criterion~\eqref{eqn:Choicestheta.1} becomes $\re^{\ri n\theta}=\re^{\ri\pi/2}$.
    There are $n-N$ integers $r$ for which $1\leq r\leq n-N$, hence
    \[
        1 < 3 \leq 4r-1 \leq 2n-3 < 2n-1.
    \]
    Setting $\theta=-\pi(4r-1)/2n$ for each of those $n-N$ choices of $r$, we obtain $n-N$ values of $\theta$ satisfying both criteria.
    Other integers $r$ violate criterion~\eqref{eqn:Choicestheta.2} and noninteger $r$ violate criterion~\eqref{eqn:Choicestheta.1}.
    
    If $n$ is even then $n-N=n/2$ and there must exist some $\phi\in[-\pi/2,\pi/2]$ for which $a=\re^{\ri\phi}$.
    Defining $\phi'=2(\pi+\phi)/\pi$, we see that $\phi'\in[1,3]$ and criterion~\eqref{eqn:Choicestheta.1} reduces to $\re^{\ri n\theta}=\re^{-\ri\phi'\pi/2}$.
    There are $n-N$ integers $r$ such that $0 \leq r \leq n-N-1$.
    For each of those integers $r$,
    \[
        1 \leq 0+\phi' \leq 4r+\phi' \leq 2n-4+\phi' \leq 2n-1.
    \]
    For each of the $n-N$ such $r$, the choice $\theta=-\pi(4r+\phi')/2n$ satisfies both criteria.
    Shifting such $\theta$ by a noninteger multiple of $4r\pi/2n$ violates criterion~\eqref{eqn:Choicestheta.1}, and any other integer choice for $r$ fails criterion~\eqref{eqn:Choicestheta.2}.
\end{proof}

\begin{figure}
    \centering
    \includegraphics{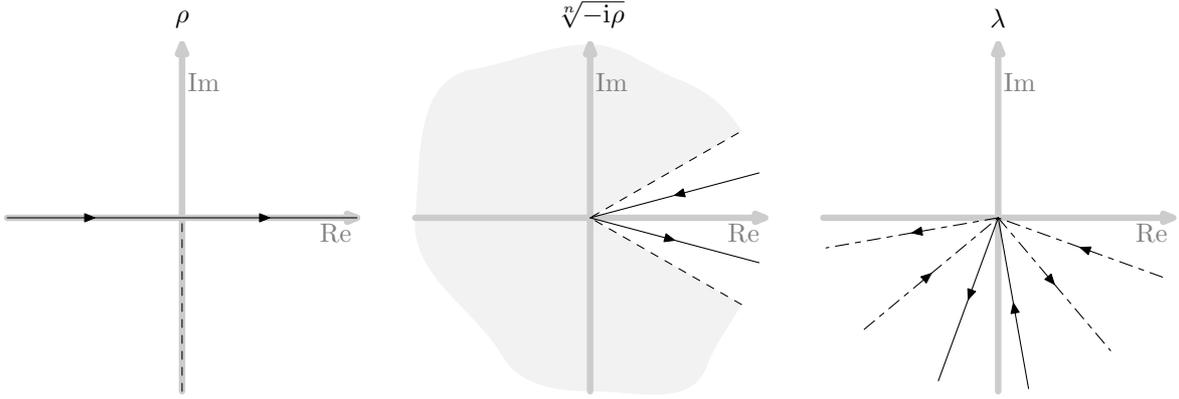}
    \caption{
        The maps $\rho\mapsto\la=\re^{\ri\theta_k}\sqrt[n]{-\ri\rho}$ in which $n=6$ and $a=\re^{\ri\pi/6}$ so that $\theta_4=-5\pi/36$, $\theta_5=-17\pi/36$, and $\theta_6=-29\pi/36$.
        The $n$\textsuperscript{th} root is biholomorphic between the unshaded regions of the first two diagrams and the function defining the map has a branch cut along the negative imaginary half axis.
    }
    \label{fig:General.maps}
\end{figure}

Lemma~\ref{lem:Choicestheta} guarantees the existence of $n-N$ maps $\la = \re^{\ri\theta}\sqrt[n]{-\ri\rho}$, $\la^n=\ri\rho/a$ for which criteria~\eqref{eqn:Choicestheta.1} and~\eqref{eqn:Choicestheta.2} both hold, as depicted in figure~\ref{fig:General.maps}.
We index $\theta=\theta_k$ by $k\in\{N+1,N+2,\ldots,n\}$.
In the global relation~\eqref{eqn:GR}, we apply each of those $n-N$ maps to obtain
\begin{equation} \label{eqn:General.GRChanged}
    \hat{q}_0\left( \re^{\ri\theta_k}\sqrt[n]{-\ri\rho} \right) - \re^{\ri\rho T} \hat{q}\left( \re^{\ri\theta_k}\sqrt[n]{-\ri\rho} ; T \right)
    = \sum_{j=0}^{n-1} c_j\left( \re^{\ri\theta_k}\sqrt[n]{-\ri\rho} \right) f_j\left( \re^{\ri\theta_k}\sqrt[n]{-\ri\rho} ; T \right)
\end{equation}
Applying the inverse Fourier transform to equation~\eqref{eqn:General.GRChanged}, and using proposition~\ref{prop:FTCaputoFiniteInterval} and lemma~\ref{lem:RemoveqT}, we obtain, for $k\in\{N+1,N+2,\ldots,n\}$,
\begin{multline} \label{eqn:General.FLODE}
    \sum_{j=0}^{n-1} c_j\left( \re^{\ri\theta_k} \right) \CaputoSeq{1/n}{(n-1-j)}{[\partial_x^jq](0,\argdot)} (t) \\
    = \frac{1}{2\pi} \int_{-\infty}^\infty \re^{-\ri\rho t} \left[ \hat{q}_0\left( \re^{\ri\theta_k}\sqrt[n]{-\ri\rho} \right) - \sum_{j=0}^{n-2} \frac{c_j\left( \re^{\ri\theta_k} \right) q_0^{(j)}(0)}{\left(\sqrt[n]{-\ri\rho}\right)^{j+1}}   \right] \D \rho.
\end{multline}
Combined with the dynamic boundary conditions~\eqref{eqn:IdBVP.dBC}, this provides a system of $n$ inhomogeneous variable coefficient sequential Caputo fractional linear ordinary differential equations in the $n$ functions $\partial_x^j q(0,\argdot)$.
We proceed, under the assumption that the system has a unique solution.

\begin{thm} \label{thm:General.convergence}
    Suppose that system~\eqref{eqn:General.FLODE}, together with the dynamic boundary conditions~\eqref{eqn:IdBVP.dBC} and compatibility conditions~\eqref{eqn:CompatCond}, has a unique solution; for the avoidance of confusion, we denote this solution by $y_j$ in place of $\partial_x^jq(0,\argdot)$.
    Suppose further that each $y_j$ is a $\frac{1}{n}$ analytic function at $0$ with partial sums $\M{y}{j}{U}$ and radius of convergence strictly greater than $T$.
    For $U\in\NN$ and $(x,t)\in[0,\infty)\times[0,T]$, let
    \[
        q_U(x,t) = \frac{1}{2\pi} \int_{-\infty}^\infty \re^{\ri\la x-a\la^nt} \hat{q}_0(\la) \D\la - \frac{1}{2\pi} \int_{\partial D_R} \re^{\ri\la x-a\la^nt} \sum_{j=0}^{n-1} c_j(\la) F[\M{y}{j}{U}](\la;T) \D\la.
    \]
    Suppose that $q$ is a solution of IdBVP~\eqref{eqn:IdBVP}.
    Then $q$ is the unique solution of IdBVP~\eqref{eqn:IdBVP} among functions satisfying~\eqref{eqn:SpaceForq} and the criteria of this theorem, and $q_U(x,t) \to q(x,t)$ as $U\to\infty$, uniformly in $t\in[0,T]$ and pointwise in $x\in[0,\infty)$.
    If also $n$ is even and $\Re(a)>0$, or $n$ is odd and $a=-\ri$, then, uniformly in $(x,t)\in[0,\infty)\times[0,T]$,
    \[
        \lim_{U\to\infty}q_U(x,t) = q(x,t).
    \]
\end{thm}

\begin{proof}
    By the above argument, if $q$ is a solution of problem~\eqref{eqn:IdBVP}, then its boundary values satisfy the same system as $y_j$.
    By hypothesis, this system has a unique solution, so $q$ must also be unique.
    
    Let $\M{R}{j}{U}$ be the tails of the fractional power series so that $y_j=\M{y}{j}{U}-\M{R}{j}{U}$.
    Then $\M{R}{j}{U}(t)\to0$ uniformly in $t\in[0,T]$ and, by linearity, the uniform limit we aim to prove is equivalent to $E_U(x,t)\to0$ uniformly in $(x,t)\in[0,\infty)\times[0,T]$, for
    \[
        2\pi E_U(x,t) = \int_{\partial D_R} \re^{\ri\la x - a\la^nt} \sum_{j=0}^{n-1} c_j(\la) F[\M{R}{j}{U}](\la;T) \D\la.
    \]
    
    It follows from its definition that $D_R$ is the disjoint union of the $N$ sectorial connected components
    \begin{equation} \label{eqn:General.sectorsDR}
        D_R = \bigcup_{k=1}^N \big\{ \la\in\CC : \lvert\la\rvert>R \mbox{ and } \tfrac{1}{n}\left(\tfrac{\pi}{2}(4k-3)-\arg(a)\right)<\arg(\la)<\tfrac{1}{n}\left(\tfrac{\pi}{2}(4k-1)-\arg(a)\right) \big\}.
    \end{equation}
    In the integral over the $k$\textsuperscript{th} connected component of $\partial D_R$, we apply the map $\la=\re^{\ri\theta_k}\sqrt[n]{-\ri\rho}$, $\la^n=\ri\rho/a$ for $\theta_k=\frac{1}{n}(\pi(2k-1)-\arg(a))$ to obtain
    \begin{multline*}
        2\pi E_U(x,t) = \\
        -\ri \int_{-\infty}^\infty \left(\sqrt[n]{-\ri\rho}\right)^{-(n-1)} \sum_{k=1}^N \re^{\ri\theta_k} \re^{\ri\re^{\ri\theta_k}\sqrt[n]{-\ri\rho} x - \ri\rho t} \sum_{j=0}^{n-1} c_j\left(\re^{\ri\theta_k}\sqrt[n]{-\ri\rho}\right) \int_0^T \re^{\ri\rho s} \M{R}{j}{U}(s) \D s \D\rho \\
        = - \int_{-\infty}^\infty \sum_{k=1}^N \re^{\ri\re^{\ri\theta_k}\sqrt[n]{-\ri\rho} x - \ri\rho t} \sum_{j=0}^{n-1} \frac{1}{\left(\ri\re^{\ri\theta_k}\sqrt[n]{-\ri\rho}\right)^{j}} \int_0^T \re^{\ri\rho s} \M{R}{j}{U}(s) \D s \D\rho \\
        = \int_{-\infty}^\infty \re^{-\ri\rho t} \sum_{k=1}^N \re^{\ri\re^{\ri\theta_k}\sqrt[n]{-\ri\rho} x} \int_0^T \re^{\ri\rho s} \M{\mathcal{R}}{k}{U}(s) \D s \D\rho,
    \end{multline*}
    in which
    \begin{equation} \label{eqn:General.convergence.defnRkU}
        \M{\mathcal{R}}{k}{U}(t) = -\sum_{j=0}^{n-1} \left(\ri\re^{\ri\theta_k}\right)^{-j} \left(\Msup{I}{0}{+}{j/n}\M{R}{j}{U}\right)(t).
    \end{equation}
    Each of the Riemann-Liouville fractional integrals of $\M{R}{j}{U}$ on the right of equation~\eqref{eqn:General.convergence.defnRkU} is the tail of a $\frac{1}{n}$ power series with the same radius of convergence strictly greater than $T$.
    Hence, for each $k\in\{1,2,\ldots,N\}$, $\M{\mathcal{R}}{k}{U}(t)\to0$ as $U\to\infty$, uniformly in $t\in[0,T]$.
    
    At $x=0$, the usual Fourier inversion theorem establishes that $E_U(0,t)=\sum_{k=1}^N \M{\mathcal{R}}{k}{U}(t)$, so $E_U(0,t)\to0$ as $U\to\infty$, uniformly in $t\in[0,T]$.
    
    For any fixed $x$,
    \[
        E_U(x,t) = \sum_{k=1}^N \left( \Phi_k(\argdot;x) * \M{\mathcal{R}}{k}{U} \right) (t),
    \]
    for
    \[
        \Phi_k(t;x) = \int_{-\infty}^\infty \re^{-\ri\rho t} \re^{\ri\re^{\ri\theta_k}\sqrt[n]{-\ri\rho} x} \D\rho.
    \]
    For any $x$, $\re^{\ri\re^{\ri\theta_k}\sqrt[n]{-\ri\rho} x}$ is a continuous and bounded function of $\rho$.
    Therefore it is a tempered distribution, and $\Phi_k$ is also a tempered distribution.
    So $E_U(x,t)$ is (a sum of) convolutions of tempered distributions with functions converging, uniformly in $t\in[0,T]$, along with all their derivatives, to $0$.
    Therefore, for each $x\geq0$, $E_U(x,t)\to0$ as $U\to\infty$ uniformly in $t\in[0,T]$.
    This completes the proof of the first convergence claim.
    
    The stronger assumption on $a$ is sufficient to ensure that, for all nonzero $\rho$, $\Re(\ri\re^{\ri\theta_k}\sqrt[n]{-\ri\rho})<0$ and, in the limit $\rho\to\pm\infty$, $\Re(\ri\re^{\ri\theta_k}\sqrt[n]{-\ri\rho})\to-\infty$.
    Therefore, the integral defining $\Phi_k$ converges as a Lebesgue integral and need not be interpreted in a distributional sense.
    By equation~\eqref{eqn:General.sectorsDR},
    \[
        0<m:=\min\left\{\arg\left(\ri\re^{\ri\theta_k}\sqrt[n]{\mp\ri}\right):k\in\{1,2,\ldots,N\}\right\}.
    \]
    We use this to bound
    \[
        \left\lVert \Phi_k(\argdot;x) \right\rVert_1 \leq \frac{2Tn!}{(mx)^n}.
    \]
    It follows that
    \[
        \left\lvert E_U(x,t) \right\rvert \leq \frac{2Tn!}{(mx)^n} \sum_{k=1}^N \left\lVert \M{\mathcal{R}}{k}{U} \right\rVert_\infty
    \]
    If $x_0>0$ is fixed, then the right converges, uniformly in $(x,t)\in[x_0,\infty)\times[0,T]$, to $0$ as $U\to\infty$.
    
    By proposition~\ref{prop:SolRep2valid}.\ref{itm:SolRep2valid.Convergence}, with $q_0=0$ and $\phi_j=\M{R}{j}{U}$, the $x$ derivative of $E_U$ is given by differentiating under the integral, so
    \[
        2\pi\partial_xE_U(0,t) = \int_{-\infty}^\infty \sqrt[n]{-\ri\rho} \re^{-\ri\rho t} \sum_{k=1}^N \ri\re^{\ri\theta_k} \int_0^T\re^{\ri\rho s}\M{\mathcal{R}}{k}{U}(s)\D s \D \rho
        = \sum_{k=1}^N \ri\re^{\ri\theta_k} \Caputo{1/n}{\M{\mathcal{R}}{k}{U}}(t).
    \]
    The latter converges as $U\to\infty$ uniformly in $t\in[0,T]$, so $E_U(x,t)$ is equicontinuous in $(t,U)\in[0,T]\times\NN$ at $x=0$.
    Hence $E_U\to0$ as $U\to\infty$ uniformly in $(x,t)\in[0,\infty)\times[0,T]$.
\end{proof}

\begin{thm} \label{thm:General.Stage3}
    Suppose that $\{(\M{y}{j}{U})_{U\in\NN}:j\in\{0,1\ldots,n-1\}\}$ is the set of partial sums whose limits are the $\frac{1}{n}$ power series $\{y_j:j\in\{0,1\ldots,n-1\}\}$ which satisfy simultaneously the system of sequential Caputo fractional differential equations~\eqref{eqn:General.FLODE} and the dynamic boundary conditions~\eqref{eqn:IdBVP.dBC}.
    Supposing that such a solution set exists, there exists a solution to IdBVP~\eqref{eqn:IdBVP}.
    Moreover, for each $k\in\{1,2,\ldots,N\}$, uniformly in $t\in[0,T]$,
    \[
        \lim_{U\to\infty}\left( \sum_{j=0}^{n-1} \M{b}{k}{j}(t) \M{y}{j}{U}(t) \right) = h_k(t),
    \]
    so that, in the limit $U\to0$, dynamic boundary conditions~\eqref{eqn:IdBVP.dBC} are recovered.
\end{thm}

\begin{proof}
    Because the solution set $\{y_j:j\in\{0,1\ldots,n-1\}\}$ satisfy equations~\eqref{eqn:General.FLODE}, we can apply the derivation of the fractional differential equations in reverse to justify the existence of a function $\gamma$ whose Fourier transform satisfies
    \[
        \hat{q}_0(\la) - \re^{a\la^nt}\hat{\gamma}(\la) = \sum_{j=0}^{n-1}c_j(\la)F[y_j](\la).
    \]
    Now proposition~\ref{prop:BVattained} implies that, if we define $q$ by
    \[
        q(x,t) = \frac{1}{2\pi}\int_{-\infty}^\infty \re^{\ri\la x-a\la^nt} \hat{q}_0(\la) \D\la - \frac{1}{2\pi}\int_{\partial D_R} \re^{\ri\la x-a\la^nt} \sum_{j=0}^{n-1} c_j(\la) F[y_j](\la) \D\la,
    \]
    then $\partial_x^jq(0,t) = y_j(t)$.
    But the solution set also satisfies the dynamic boundary conditions, so also must $q$.
    
    By construction, $q$ differs from $q_U$ only by $E_U$, in which
    \[
        2\pi E_U(x,t) = \int_{\partial D_R} \re^{\ri\la x-a\la^nt} \int_0^T \re^{a\la^ns} \sum_{j=0}^{n-1}c_j(\la)\M{R}{j}{U} \D s \D\la,
    \]
    and $\M{R}{j}{U}$ is the $U$ tail of $y_j$:
    \[
        y_j=\M{y}{j}{U}+\M{R}{j}{U}.
    \]
    So it is equivalent to prove that, for each $k\in\{1,2,\ldots,N\}$ and uniformly in $t\in[0,T]$, $\M{\mathcal{E}}{k}{U}(t)\to0$ as $U\to\infty$, where
    \[
        \M{\mathcal{E}}{k}{U}(t) = \int_{\partial D_R} \re^{-a\la^nt} \sum_{r=0}^{n-1}(\ri\la)^r\M{b}{k}{r}(t) \int_0^T \re^{a\la^ns} \sum_{j=0}^{n-1}c_j(\la)\M{R}{j}{U} \D s \D\la.
    \]
    The proof now proceeds along the lines of the proof of theorem~\ref{thm:HeatStage3}, modified slightly as was the proof of theorem~\ref{thm:General.convergence} from the proof of theorem~\ref{thm:Heat.Convergence}.
\end{proof}

\section{Conclusion} \label{sec:Conclusion}

We have studied half line initial dynamic boundary value problems of class~\eqref{eqn:IdBVP}, with arbitrary spatial order and arbitrary linear boundary conditions.
We have detailed a method for obtaining any existing solution of such a problem, via contour integrals of fractional power series, and justified that the error inherrent in approximating the power series by their partial sum decays uniformly.
We also gave a criterion for establishing existence.
For each of the specific example problems~\eqref{eqn:HeatIdBVP}--\eqref{eqn:LKdV2IdBVP}, we have provided a recurrence relation for the coefficients in the fractional power series, thus explicitly solving each IdBVP.

\begin{rmk}
    There are certainly situations in which it is not possible to solve the system given by equations~\eqref{eqn:IdBVP.dBC} and~\eqref{eqn:General.FLODE}.
    An example that may occur even for static boundary conditions is failure of linear independence of the (dynamic) boundary conditions.
    Even beyond such algebraic considerations, the authors are not aware of a higher order systems fractional Frobenius theory for sequential Caputo fractional derivatives in the vein of~\cite[\S7.5]{KST2006a}, even at ordinary points.
    
    Classification of systems of dynamic boundary conditions yielding well posed problems, comparable to~\cite{Smi2012a}, along with the underlying sequential Caputo fractional Frobenius systems theory, is left for future work.
    There is a subtle question to investigate here: can an IdBVP be well posed even if at some instant $t\in[0,T]$ the dynamic boundary conditions are not linearly independent?
\end{rmk}

\begin{rmk}
    The present work generalises the Fokas transform method to IdBVP on the half line.
    The Fokas transform method has been successfully applied to static IBVP on the finite interval~\cite{FP2001a,Smi2012a}, and also to interface~\cite{DPS2014a,SS2015a,DS2015a,DSS2016a,DS2020a,STV2019a}, multipoint~\cite{FM2013a,PS2018a}, and nonlocal BVP~\cite{MS2018a}, so one might expect that the present results could be generalised to study also IdBVP on such spatial domains.
    Unfortunately, this will be a more complicated task than was the generalisation from half line to finite interval IBVP.
    Indeed, a dynamic boundary condition for an equation on the finite interval does not reduce the global relation to a fractional ordinary differential equation, but to a rather more complicated pseudodifferential equation.
\end{rmk}

\begin{rmk}
    Spatially lower order terms are absent from~\eqref{eqn:IdBVP.PDE}.
    It may be possible to extend the arguments and results of this work to accommodate such equations.
    Although the Ehrenpreis form usually associated with such problems does not immediately evoke Fourier transforms of fractional differential operators, under the change of variables in~\cite{DS2020a,STV2019a}, it may be possible to recover fractional differential equations.
\end{rmk}

\begin{rmk}
    The classical Fourier $b$ transform method (see, for example,~\cite[\S5.1.4]{Pin2011a}) is a synthesis of the spatial Fourier transform with the method of images, tailored to static Robin problems.
    It can be used to solve problem~\eqref{eqn:HeatIdBVP} provided $b$ is constant.
    Indeed, for $F_{\mathrm{c}}$ and $F_{\mathrm{s}}$ the usual Fourier cosine and sine transforms, the Fourier $b$ transform is defined as
    \[
        \widetilde{F}[\phi](\la) = \frac{\ri\la F_{\mathrm{c}}[\phi](\la)-\ri b F_{\mathrm{s}}[\phi](\la)}{2(b+\ri\la)}
    \]
    and has an inversion formula valid wherever the Fourier sine and cosine inversion formulae are valid.
    This transform has the property that $\widetilde{F}[-\phi''](\la)=\la^2 \widetilde{F}[\phi](\la)$ for all functions $\phi$ for which $\phi,\phi',\phi''\in L_1[0,\infty)$ and $b\phi(0)+\phi'(0)=0$.
    Therefore, applying this transform to~\eqref{eqn:HeatIdBVP.PDE}, one obtains
    \[
        \widetilde{F}[\partial_tq(\argdot,t)](\la) + \la^2 \widetilde{F}[q(\argdot,t)](\la)=0.
    \]
    Under a mild regularity assumption, it is possible to interchange the spatial Fourier $b$ transform with the temporal partial derivative to obtain a temporal ordinary differential equation for $\widetilde{F}[q(\argdot,t)](\la)$ valid for all $\la$.
    Hence, one proceeds in the usual fashion to solve that differential equation and apply the inverse transform to yield a solution of the original problem.
    The step that fails for time dependent $b$ is the interchange of transform and derivative.
    Although the transform is always in the spatial variable, if $b$ is time dependent, then the kernel of the transform is time dependent, so $\widetilde{F}[\partial_tq(\argdot,t)](\la)\neq\partial_t\widetilde{F}[q(\argdot,t)](\la)$, and the argument cannot proceed in the usual way.
    
    The Fokas transform method is often seen as a tool for deriving the precise spatial transform pair tailored to solving a particular IBVP: the transform pair for which any function in the kernel of the boundary form will yield no boundary term in the relevant transform of derivative formula.
    Therefore, it would be interesting to attempt to save the Fourier $b$ transform argument, perhaps using some kind of integrating factor, to find whether it can yield a similar solution representation to that presented above.
\end{rmk}

\begin{rmk}
    From a spectral theoretic viewpoint, the Fokas transform method employs the transform pair that diagonalises the ordinary differential operator that describes the spatial part of an IBVP~\cite{Smi2015a,PS2016a}.
    It would be interesting to understand how this spectral perspective applies to IdBVP in which the spatial ordinary differential operator is time dependent, albeit formally independent of time.
    The results for IBVP extend to finite interval problems~\cite{FS2016a,ABS2020a}, so this may also be an avenue for exploring an extension of the Fokas transform method for IdBVP to finite interval IdBVP.
\end{rmk}

\section*{Acknowledgement}

Smith would like to thank the Isaac Newton Institute for Mathematical Sciences for support and hospitality during programme \emph{Complex analysis: techniques, applications and computations}, when work on this paper was undertaken.
This work was supported by EPSRC Grant Number EP/R014604/1.
Smith gratefully acknowledges support from Yale-NUS College project B grant IG18-PRB102.
Toh was partially supported by the Yale-NUS College Summer Research Programme.

\appendix

\section{Fourier transforms of Caputo fractional derivatives} \label{sec:CaputoFourier}

We recall the definitions and theorems associated with Caputo fractional derivatives and their Fourier transforms.

\begin{defn} \label{defn:Caputo}
    For $\alpha\in(0,1]$, the \emph{(left sided half axis) Riemann-Liouville fractional integral operator} $\Msup{I}{0}{+}{\alpha}$ is defined by
    \[
        \left( \Msup{I}{0}{+}{\alpha} y \right) (x) := \frac{1}{\Gamma(\alpha)} \int_0^x \frac{y(t)}{(x-t)^{1-\alpha}}\D t.
    \]
    wherever the integral converges.
    For $\alpha\in(0,1)$, the \emph{(left sided half axis) Caputo fractional derivative operator} $\Caputo{\alpha}{y}$ is defined by
    \[
        \Caputo{\alpha}{y}(x) = \left( \Msup{I}{0}{+}{1-\alpha} \frac{\D}{\D x} y \right) (x),
    \]
    wherever it converges.
    For nonnegative integers $n$, the $n$th \emph{sequential (left sided half axis) Caputo fractional derivative operator} $\CaputoSeq{\alpha}{n}{y}$ is defined inductively by $\prescript{\ensuremath{\mathrm{C}}}{}{\ensuremath{\mathrm{D}}}_{0\Mspacer+}^{\alpha\Mspacer1} = \prescript{\ensuremath{\mathrm{C}}}{}{\ensuremath{\mathrm{D}}}_{0\Mspacer+}^{\alpha}$
    and, for $n\geq2$,
    $\prescript{\ensuremath{\mathrm{C}}}{}{\ensuremath{\mathrm{D}}}_{0\Mspacer+}^{\alpha\Mspacer{n}} = \prescript{\ensuremath{\mathrm{C}}}{}{\ensuremath{\mathrm{D}}}_{0\Mspacer+}^{\alpha} \prescript{\ensuremath{\mathrm{C}}}{}{\ensuremath{\mathrm{D}}}_{0\Mspacer+}^{\alpha\Mspacer{n-1}}$.
    For notational convenience, we define $\prescript{\ensuremath{\mathrm{C}}}{}{\ensuremath{\mathrm{D}}}_{0\Mspacer+}^{\alpha\Mspacer0}$ as the identity operator.
\end{defn}

\begin{defn} \label{defn:FTforCaputo}
    We define the \emph{Fourier transform} $\left(\mathcal{F}y\right)$ of any $y\in L_1[0,T]$ by
    \[
        \left(\mathcal{F}y\right)(\la) = \int_0^T \re^{\ri\la t} y(t) \D t,
    \]
    for all $\la\in\CC$.
\end{defn}

Note that, for this Fourier transform, if $y$ is an absolutely continuous function on $[0,T]$, then integration by parts yields that
\begin{equation} \label{eqn:FourierDerivativeFiniteInterval}
    \left(\mathcal{F}y'\right)(\la) = (-\ri\la)\left(\mathcal{F}y\right)(\la) - y(0) + y(T)\re^{\ri\la T},
\end{equation}
for all $\la\in\CC$.

\begin{prop} \label{prop:FTRLFracInt}
    Suppose that $\alpha\in(0,1)$ and $y\in L_1[0,T]$.
    Then, for all $\la\in\CC$,
    \begin{equation} \label{eqn:FTCaputo.Int}
        \left( \mathcal{F} \Msup{I}{0}{+}{\alpha} y \right)(\la) =  (-\ri\la)^{-\alpha} \left(\mathcal{F}y\right)(\la).
    \end{equation}
\end{prop}

\begin{proof}
    We denote also by $y$ its zero extension to $\RR$.
    For such a function $y$, the Fourier transform of definition~\eqref{defn:FTforCaputo} coincides with the full line Fourier transform.
    Moreover, on the relevant domain $t\in[0,T]$, the left sided ``half axis'' fractional integral of $y$ defined in definition~\ref{defn:Caputo} coincides with the left sided ``full axis'' fractional integral of $y$ in which the integral begins at $-\infty$ instead of $0$.
    Therefore, this result is a corollary of~\cite[theorem~7.1]{KMS2002a}.
\end{proof}

\begin{prop} \label{prop:FTCaputoFiniteInterval}
    If $y$ is absolutely continuous on $[0,T]$, and zero elsewhere, then
    \begin{equation} \label{eqn:FTCaputo.Deriv}
        \left(\mathcal{F} \prescript{\ensuremath{\mathrm{C}}}{}{\ensuremath{\mathrm{D}}}_{0\Mspacer+}^{\alpha} y \right)(\la) = (-\ri\la)^\alpha \left(\mathcal{F}y\right)(\la) - (-\ri\la)^{-(1-\alpha)}\left( y(0) - y(T)\re^{\ri\la T} \right).
    \end{equation}
\end{prop}

\begin{proof}
    The proof is immediate from definition~\ref{defn:Caputo}, proposition~\ref{prop:FTRLFracInt} and equation~\eqref{eqn:FourierDerivativeFiniteInterval}.
\end{proof}

\bibliographystyle{amsplain}
{\small\bibliography{../dbrefs}}

\end{document}